\documentclass[12pt,reqno,a4paper]{amsart}
\usepackage{mdframed}
\usepackage{setspace}
\usepackage{mathrsfs}  
\usepackage{comment}

\usepackage[utf8]{inputenc}
\usepackage[T1]{fontenc}

\usepackage{hyperref}

\usepackage{bold-extra}
\usepackage{etoolbox}
\usepackage{amsmath}
\usepackage{amssymb}
\usepackage{amscd}
\usepackage{amsthm}
\usepackage{amsfonts}
\usepackage{mathtools}
\usepackage{graphicx}
\usepackage[all,cmtip]{xy}
\usepackage[shortlabels]{enumitem}
\usepackage{todonotes}
\usepackage{bbm}
\usepackage{tikz-cd}
\usepackage{array}
\usepackage{booktabs}
\usepackage{caption}

\patchcmd{\subsection}{-.5em}{.5em}{}{}
\patchcmd{\subsubsection}{-.5em}{.5em}{}{}

\theoremstyle{theorem}
\newtheorem{theorem}{Theorem}[section]
\newtheorem{lemma}[theorem]{Lemma}
\newtheorem{proposition}[theorem]{Proposition}
\newtheorem{corollary}[theorem]{Corollary}

\newtheorem{mainthm}{Theorem}

\theoremstyle{definition}
\newtheorem{definition}[theorem]{Definition}
\newtheorem{example}[theorem]{Example}
\newtheorem{remark}[theorem]{Remark}
\newtheorem*{remark*}{Remark}

\numberwithin{theorem}{section}
\numberwithin{equation}{section}

\newcommand{\SO}{\operatorname{SO}}

\newcommand{\R}{\mathbb{R}}
\newcommand{\N}{\mathbb{N}}
\newcommand{\Z}{\mathbb{Z}}
\newcommand{\C}{\mathbb{C}}
\newcommand{\UD}{\mathrm{UD}}

\newcommand{\setdiv}{\ \middle|\ }

\newcommand{\supp}{\mathrm{supp}}
\newcommand{\mycomment}[1]{}

\renewcommand{\epsilon}{\varepsilon}

\DeclarePairedDelimiter{\norm}{\lVert}{\rVert}
\DeclarePairedDelimiter{\abs}{\lvert}{\rvert}

\DeclareMathOperator{\lspan}{span}

\allowdisplaybreaks

\begin{document}

\title[Linear programming bounds in homogeneous spaces, I]{Linear programming bounds in homogeneous spaces, I: Optimal packing density}
\author{Maximilian Wackenhuth}
\address{Institut f\"ur Algebra und Geometrie, KIT, Karlsruhe, Germany}
\email{maximilian.wackenhuth@kit.edu}

\date{}

\begin{abstract}
	In this article we obtain linear programming bounds for the maximal sphere packing density of commutative spaces.
	A special case of our results solves a conjecture by Cohn and Zhao on linear programming bounds for sphere packings in hyperbolic space.
\end{abstract}

\maketitle



\section{Introduction}
\subsection{Euclidean linear programming bounds}
The linear programming approach by Cohn and Elkies to the density of sphere packings, developed in \cite{CohnElkies2003}, has recently attracted much attention due to its role in the solution of the packing problem in dimension 8 due to Viazovska in \cite{ViazovskaDim8} and in dimension 24 due to Cohn, Kumar, Miller, Radchenko and Viazovska in \cite{ViazovskaDim24}.
Recall that the \emph{optimal packing density} of $\R^n$ is defined as
\begin{align}\label{e:definition_euclidean_optimal_density}
	\triangle(r, \R^n)\coloneqq  \sup_P \limsup_{R\to \infty}\frac{\lambda(B(0,R)\cap \bigcup_{B\in P}B)}{\lambda(B(0,R))},
\end{align}
where the supremum is taken over all $r$-sphere packings in $\R^n$, i.e. all sets of disjoint open balls of radius $r$.
Cohn and Elkies have shown that
\begin{align}\label{e:CohnElkiesBound}
	\triangle(r,\R^n)\leq \lambda(B(0,r))\frac{f(0)}{\widehat{f}(0)}
\end{align}
for all $f$ in a certain set $\mathcal{W}(r,\R^n)$ of witness functions. Here a function $f:\R^n\to \R$  is called a \emph{witness function} if 
\begin{enumerate}[(W1)]
	\item $f(x)\leq 0$ if $\norm{x}\geq 2r$,
	\item $\widehat{f}\geq 0$ and $\widehat{f}(0)>0$,
	\item $f$ satisfies suitable decay and smoothness conditions.
\end{enumerate}
A possible choice for the condition (W3) is that $\abs{f(x)}$ and $\abs{\widehat{f}(x)}$ are bounded above by $(1+\abs{x})^{-n-\delta}$ for some $\delta>0$. The results of \cite{ViazovskaDim8} and \cite{ViazovskaDim24} are then established by constructing explicit witness functions, which match lower bounds coming from specific lattice packings. The proof of \eqref{e:CohnElkiesBound} proceeds in two steps. In the first step Cohn and Elkies use the Poisson summation formula to bound the density of periodic sphere packings. In the second step they use that the optimal packing density of $\R^n$ can be approximated by densities of periodic packings.

\subsection{Hyperbolic linear programming bounds}

This article is concerned with generalizing the linear programming bound \eqref{e:CohnElkiesBound} to a large class of non-Euclidean geometries. 
A sample application of our method concerns packing bounds for hyperbolic space $\mathbb{H}^n$. A version of these bounds was conjectured by Cohn and Zhao in \cite{CohnZhao}.

Choose a basepoint  $x_0\in \mathbb{H}^n$ and denote by  $m_{\mathbb{H}^n}$ the measure on $\mathbb{H}^n$ induced by the Riemannian metric. Bowen and Radin \cite{Bowen2003, Bowen2004} have identified a class of sphere packings in hyperbolic space with well-behaved densities. We refer to these as \emph{generically measured sphere packings}, see Definition \ref{d:generically_measured} below.
By analogy with \eqref{e:definition_euclidean_optimal_density} they define
\[\triangle(r,\mathbb{H}^n)\coloneqq \sup_P\limsup_{R\to \infty} \frac{m_{\mathbb{H}^n}(B(x_0,R)\cap \bigcup_{B\in P}B)}{m_{\mathbb{H}^n}(B(x_0,R))},\]
where the supremum is taken over all generically measured $r$-sphere packings $P$ in $\mathbb{H}^n$.

To state our bound let us call a function $f:\mathbb{H}^n\to \R$ of the form $f(x) = h(d(x,x_0))$ a \emph{witness function of radius $r$} if
\begin{enumerate}[(W1)]
	\item $f(x)\leq 0$ if $d(x,x_0)\geq 2r$,
	\item $\widehat{f}\geq 0$ and $\widehat{f}(\mathbf{1})>0$,
	\item $h\in \cosh^{-n+1}\mathscr{S}(\R)$ is even.
\end{enumerate}
Here, $\widehat{f}$ denotes the spherical Fourier transform of $f$ and $\mathbf{1}$ denotes the trivial character on the isometry group of $\mathbb{H}^n$.
We denote by $\mathcal{W}(r,\mathbb{H}^n)$  the class of witness functions of radius $r$.

In this article we prove the following version of Cohn and Zhao's conjecture. A sketch of a proof of this theorem was given by us in \cite{Wackenhuth1}:
\begin{mainthm}\label{t:mainthm1}
	For all $f\in\mathcal{W}(r, \mathbb{H}^n)$ we have
	\[\triangle(r,\mathbb{H}^n) = m_{\mathbb{H}^n}(B(x_0,r))\frac{f(x_0)}{\widehat{f}(\mathbf{1})}.\]
\end{mainthm}

\begin{remark}
	\begin{enumerate}[(i)]
		\item For periodic packings, a version of Theorem \ref{t:mainthm1} was established by Cohn and Zhao \cite{CohnZhao}, based on a pre-trace formula by Cohn, Lurie and Sarnak. This implies the theorem in dimension $n=2$, since Bowen has shown in \cite{BowenPeriodicityDim2} that the optimal packing density of the hyperbolic plane can be approximated by densities of periodic sphere packings.
		\item It is a notorious open problem whether the optimal packing density in higher dimensions can be approximated by densities of periodic sphere packings. Bowen's approach might generalize to dimension $3$, but certainly not to higher dimensions. Our approach does not use any approximation by periodic packings.
		\item Cohn and Zhao conjectured the bound in Theorem \ref{t:mainthm1}, but with the weaker assumption that $f$ is continuous and integrable. They proved that their conjectured bound always beats the hyperbolic version  of the Kabatiansky--Levenshtein bound.
	\end{enumerate}
\end{remark}
For the reader's convenience, we now spell out Theorem \ref{t:mainthm1} in a way that requires no knowledge of the spherical transform.
We obtain the following bounds by combining Theorem \ref{t:mainthm1} and Remark \ref{r:harmonic_analysis_hyperbolic_space} below.
\begin{example}\label{e:hyperbolic_bounds}
	For $n\in \N$ set $\rho=\frac{n-1}{2}$ and let $r>0$. Define a constant $C(r)>0$ by 
	\[C(r)\coloneqq \int_0^r\sinh(t)^{2\rho}dt.\]
	Let $h\in \cosh^{-2\rho}\mathscr{S}(\R)$ be even and satisfy
	\begin{enumerate}[(i)]
		\item $h(t)\leq 0$ if $\abs{t}\geq 2r$,
		\item for all $\lambda\in \R\cup i[-\rho,\rho]$ \[\int_0^\infty h(t)~_2F_1\left(\frac{\rho+i\lambda}{2},\frac{\rho-i\lambda}{2};\frac{n}{2};-\sinh(t)^2\right)\sinh(t)^{2\rho}dt \geq 0,\]
		where $~_2F_1$ denotes Gauss hypergeometric function,
		\item $\int_0^\infty h(t)\sinh(t)^{2\rho}dt >0$.
	\end{enumerate} 
	Then
	\[\triangle(r,\mathbb{H}^n)\leq C(r)\frac{h(0)}{\int_0^\infty h(t)\sinh(t)^{2\rho}dt}.\] 
\end{example}

\subsection{Beyond hyperbolic space}

Theorem \ref{t:mainthm1} is a special case of a more general theorem. In the body of this article we will work in the following general setting:
\begin{itemize}
	\item $G$ is a Lie group and $K<G$ is a compact subgroup such that $(G,K)$ is a Gelfand pair (see Section \ref{s:general assumptions}).
	\item $d$ is a $G$-invariant, complete, proper and continuous metric on $X := G/K$ such that $(G,K,d)$ satisfies an invariant pointwise ergodic theorem (see Definition \ref{d:invariant_pointwise_ergodic_theorem}).
	\item $\mathcal{S}(G,K)$ is a choice of Schwartz-like function space for $(G,K)$ (see Definition \ref{d:Schwartz-like})
\end{itemize}
We then refer to $(G,K, d, \mathcal{S}(G, K))$ (or just $(G,K)$) as a \emph{convenient Gelfand pair}. 

From now on let $(G,K, d, \mathcal{S}(G, K))$ be a convenient Gelfand pair; we fix  Haar measures on $G$ and $K$ (the latter normalized to total volume $1$) and denote by $m_X$ the corresponding $G$-invariant measure on $X = G/K$. As in the hyperbolic case, there is then a natural notion of Bowen-Radin optimal packing densities $\triangle(r, X)$ (see \ref{d:Bowen_Radin_optimal_density}), as well as a natural notion of spherical transform for functions $f\in \mathcal{S}(G, K)$, which we denote by $\widehat{f}$. We further denote by $x_0\coloneqq eK$ the base point of $X$.

\begin{definition}
	We define the space $\mathcal{W}(r, G/K)$ of \emph{witness functions} on $G/K$ as the set of functions $f:G\to \R$ such that 
\begin{enumerate}[(W1)]
	\item $f(g)\leq 0$ if $d(gx_0, x_0)\geq 2r$,
	\item $\widehat{f}\geq 0$ and $\widehat{f}(\mathbf{1})>0$,
	\item $f\in \mathcal{S}(G,K)$.
\end{enumerate} 

\end{definition}
We obtain the following general theorem:
\begin{mainthm}\label{t:mainthm2}
	If $(G,K, d, \mathcal{S}(G, K))$ is a convenient Gelfand pair, then 
	\[\triangle(r,G/K)\leq m_{G/K}(B(x_0, r))\frac{f(e)}{\widehat{f}(\mathbf{1})} 
	\quad \text{for all } f\in \mathcal{W}(r,G/K).\]
\end{mainthm}
Here, $\mathbf{1}$ denotes the trivial character on $G$ as before. Both Cohn and Elkies linear programming bound and and Theorem \ref{t:mainthm1} are special cases: 
\begin{itemize}
\item For the LP bound by Cohn and Elkie we choose $G=\R^n$, $K=\{0\}$ and let $d$ be the usual Euclidean metric. As $\mathcal{S}(G, K)$ we choose the ordinary Schwartz space $\mathscr{S}(\R^n)$. The availability of an invariant pointwise ergodic theorem follows directly from the ergodic theory of amenable groups. Hence $(\R^n, \{0\}, d, \mathscr{S}(\R^n))$ is a convenient Gelfand pair. Now Theorem \ref{t:mainthm2} recovers their bound up to a difference in function spaces. 
\item For the hyperbolic bound we choose $G=\SO(n,1)$, $K= \SO(n)$ and let $d$ denote the usual hyperbolic metric on $G/K$. As $\mathcal{S}(G, K)$ we choose the radial Harish-Chandra $L^1$-Schwartz space on $G$, which is given by $\cosh^{-n+1}\mathscr{S}(\R)$ in radial coordinates (i.e. using the Cartan decomposition $G=KA_+K$). The availability of an invariant pointwise ergodic theorem follows directly from the work of Gorodnik and Nevo, \cite{GorodnikNevo}, on the ergodic theory of semisimple Lie groups. Now Theorem \ref{t:mainthm2} applied to convenient Gelfand pair $(\SO(n,1), SO(n), d, \mathcal{S}(G, K))$ implies Theorem \ref{t:mainthm1}.
\end{itemize}

\subsection{Riemannian symmetric spaces}
Our results concerning hyperbolic spaces admit the following natural generalization to Riemannian symmetric spaces: 

If $(X,d)$ is an irreducible Riemannian symmetric space of noncompact type, there is a simple Lie group $G$ with finite center and no compact factors such that for any maximal compact subgroup $K$ of $G$ the space $X$ can be identified with $G/K$.
In this case a possible choice of Schwartz-like space is given by the set $\mathscr{S}^1(G,K)$ of bi-$K$-invariant Harish-Chandra $L^1$-Schwartz functions on $K$. 
The work of Gorodnik and Nevo, \cite{GorodnikNevo}, provides an invariant pointwise ergodic theorem for $(G, K, d)$. Hence $(G, K, d, \mathscr{S}^1(G,K))$ is a convenient Gelfand pair to which Theorem \ref{t:mainthm2} applies.

If $X$ is a rank $1$ symmetric space of noncompact type, \ref{t:mainthm2} implies bounds which are similar to Theorem \ref{t:mainthm1}.
In contrast, if $X$ is a higher rank irreducible symmetric space of noncompact type, it becomes quite difficult to write down Condition (W2) explicitly. By considering the radial part of the Laplace operator on $X$, it is possible to relate spherical functions to eigenfunctions problems of certain differential–difference operators on Euclidean spaces. Natural generalizations of these issues and related special functions can be found in the study of Dunkl theory. 

If $(X,d)$ is Riemannian symmetric space of compact type, there is a compact simple Lie group $G$ and a compact subgroup $K$ such that $X=G/K$. In this case the natural choice for $\mathcal{S}(G, K)$ is given by $C^\infty(G, K)$, the set of smooth bi-$K$-invariant functions on $G$. As compact groups are amenable, the necessary ergodic theorems follow directly from the ergodic theory of amenable groups. Hence in this case $(G, K, d, C^\infty(G, K))$ is a convenient Gelfand pair. Note that the compact Riemannian two-point homogeneous spaces are exactly the compact Riemannian symmetric spaces of rank $1$. Hence Theorem \ref{t:mainthm2} implies a variant of the well-known linear programming bounds for codes in compact two-point homogeneous spaces.

\subsection{Heisenberg groups}
	The $(2n+1)$-dimensional Heisenberg group $\mathcal{H}_n$ is given by the set $\mathcal{H}_n+\R\times C^n$ with the composition $\ast$ defined by
	\[
	(t,x)\ast (s, y)\coloneqq (t+s -\frac{1}{2}\mathrm{Im}\langle x,y \rangle,x+y).
	\]
	We equip $\mathcal{H}_n$ with the Cygan-Koranyi metric $d_{CK}$, which can be obtained by defining the Cygan-Koranyi group norm
	$\norm{(t,x)}\coloneqq (t^2+\norm{x}_2^4)^{1/2}.$
	$\mathcal{H}_n$ can be considered as the homogeneous space of the Gelfand paair $(\mathcal{H}_n\rtimes U(n), U(n))$.
	The availability of an invariant pointwise ergodic theorem for $\mathcal{H}_n\rtimes U(n)$ follows directly from the ergodic theory of amenable groups and as a Schwartz-like function space we can choose the set of radial Schwartz functions $\mathscr{S}(\mathcal{H}_n\rtimes U(n), U(n))$ on $\mathcal{H}_n\rtimes U(n)$. Then $(\mathcal{H}_n\rtimes U(n), U(n), d_{CK}, \mathscr{S}(\mathcal{H}_n\rtimes U(n), U(n)))$ is a convenient Gelfand pair.
	
	We obtain the following explicit formulas by combining Theorem \ref{t:mainthm2} with Theorem \ref{t:parametrization_spherical_heisenberg} and rewriting everything in radial coordinates.
	Assume that $h\in \mathscr{S}(\R\times \R)$ is even in the second coordinate and satisfies
	\begin{enumerate}[(i)]
		\item $h(t,s)\leq 0$ if $(t^2+s^4)^{1/2}\geq 2r$,
		\item for all $\lambda>0$ and $m\in \N$ we have 
		\[\int_\R\int_0^\infty h(t,s)e^{\pm i\lambda t}L_{m}^{(n-1)}\left(\frac{1}{2}\lambda s^2\right)e^{-\frac{1}{4}\lambda s^2}dsdt\geq 0,\]
		where $L_m^{(n-1)}$ denotes the generalized Laguerre polynomial of order $n-1$ normalized to $1$ at $0$,
		\item for all $\tau>0$ we have
		\[\int_\R\int_0^\infty h(t,s)\frac{2^{n-1}(n-1)!}{(\tau s)^{n-1}}J_{n-1}(\tau s)dsdt\geq 0,\]
		where $J_{n-1}$ denotes the Bessel function of order $n-1$,
		\item $\int_\R\int_0^\infty h(t,s)dsdt>0$.
	\end{enumerate}
	Then
	\[\triangle(r,\mathcal{H}_n) \leq C_nr^{2d+2}\frac{\Gamma(n)}{2\pi^n}\frac{h(0,0)}{\int_\R\int_0^\infty h(t,s)dtds},\]
	where $C_n$ is the volume of the ball $B(0,1)$ in the Cygan-Koranyi metric.

\subsection{Methods and proof}

Since periodic approximation is not known to hold in the generality of Theorem \ref{t:mainthm1} (and certainly not in the generality of Theorem \ref{t:mainthm2}), our approach does not work by reduction to the periodic case. Instead we reduce to the study of random invariant sphere packings, following an idea by Bowen and Radin. Our strategy is to apply recent advances in stochastic geometry on homogeneous spaces to Bowen and Radin's notion of density of random invariant sphere packings in order to prove Theorem \ref{t:mainthm2}.

More precisely, let $\UD_{2r}(X)$ denote the set of subsets $P\subset X$ such that $d(x,y)\geq 2r$ for all $x\neq y\in P$. This set can be equipped with a locally compact second countable Hausdorff metric and the canonical $G$-action on $\UD_{2r}(X)$ is continuous with respect to this metric. Let $(\Omega,\mathbb{P})$ denote a probability space equipped with a probability measure preserving $G$-action and let $\Lambda:(\Omega,\mathbb{P})\to \UD_{2r}(X)$ be a stationary point process, i.e.\ a $G$-equivariant measurable map. Then $\Lambda$ is the (random) set of centers of some random invariant $r$-sphere packing $\Lambda^r$, and the Bowen-Radin density of $\Lambda^r$ is defined as \[D(\Lambda^r) = \mathbb{P}(x_0\in \Lambda^r).\] This is related to the intensity $i(\Lambda)$ of $\Lambda$ by the formula $i(\Lambda)= \frac{D(\Lambda^r)}{m_X(B(x_0,r))}$, and coincides with the density of a generic instance of $\Lambda$. Since generically measured sphere packings are precisely the generic instances of such point processes, the proof of our main theorem thus reduces to estimating intensities of certain point processes.

By recent work of Björklund and Byléhn in \cite{MichaelMattiasArXiv}, building on work by Björklund, Hartnick and Pogorzelski in \cite{BHP3} and Björklund and Hartnick in \cite{BjorklundHartnickHyperuniformity}, we can associate two positive-definite Radon measures $\eta_\Lambda^+$ and $\eta_\Lambda\coloneqq \eta_\Lambda^+-i(\Lambda)^2m_G$ on $G$ to $\Lambda$.
As these measure are positive-definite, they have positive spherical transforms $\widehat{\eta}_\Lambda^+$ and $\widehat{\eta}_\Lambda$. Hence, if $f$ is a nice function on $G$ with $\widehat{f}\geq 0$, we have 
\[\widehat{\eta}_\Lambda^+(\widehat{f}) = \widehat{\eta}_\Lambda(\widehat{f})+i(\Lambda)^2\widehat{m}_G(\widehat{f})\geq i(\Lambda)^2\delta_{\mathbf{1}}(\widehat{f}).\]
By \cite{MichaelMattiasArXiv} the measure $\eta_\Lambda^+$ is given by the averaged summation formula
\[\eta_\Lambda^+(f)=\mathbb{E}\left[\sum_{x\in \Lambda}\sum_{y\in \Lambda}f(\sigma(x)^{-1}\sigma(y))b(x)\right],\]
where $b:X\to \R$ is a non-negative bounded measurable function with compact support and $m_x(b)=1$ and $\sigma:X\to G$ is a Borel section of the quotient map $G\to X$.
These two facts allow us to replace the use of the Poisson summation formula in proof of the Euclidean linear programming bound by Cohn and Elkies in \cite{CohnElkies2003} with the formula
\begin{align}\label{e:Poisson_replacement}\eta_\Lambda^+(f) = \widehat{\eta}^+_\Lambda(\widehat{f}),\end{align}
obtained from the Plancherel--Godement theorem.

While this argument establishes our main theorems for some class of witness functions, the resulting function space is quite inconvenient. It it thus important to enlarge the function space in question to a suitable Schwartz-like space. Here the following major technical problem arises:
While the measure $\eta_\Lambda^+$ induces a positive-definite distribution $T_\Lambda^+$ which in our setting can be extended to a ``tempered'' distribution (in the dual of our Schwartz-like space), this tempered distribution is no longer given by integration. Thus, to establish our main theorem, we need to approximate Schwartz functions $f$ by smooth compactly supported functions such that the approximants have the same sign change behaviour as $f$.

\subsection{Organization of the article}
In Section \ref{s:harmonic_analysis} we review the necessary background on harmonic analysis of functions, measures and distributions on Gelfand pairs. We define the notion of a Schwartz-like function space and recall detailed information about each of the special classes of Gelfand pairs we consider in this article.

In Section \ref{s:PointProcesses} we cover the necessary prerequisites in stochastic geometry on homogeneous spaces.

In Section \ref{s:deterministic_and_random_density} we recall Bowen and Radin's definition of packing density and explain its relation to the density of deterministic packings. We reformulate Bowen and Radin's definition in a way that relates it to general concepts in the theory of point processes.

In Section \ref{s:bounds} we prove our main result Theorem \ref{t:mainthm2}.

In Appendix \ref{S:Bochner-Schwarz appendix}, we prove a Bochner-Schwartz theorem for spherical distributions on the Heisenberg group.

\section{Background in spherical harmonic analysis}\label{s:harmonic_analysis}

In this section we recall background material about the spherical harmonic analysis of functions, measures and distributions on Gelfand pairs. We cover several examples of classes of Gelfand pairs in detail and define the notion of a Schwartz-like space of functions.
We will also use the opportunity to fix notation and to make some standing assumptions.
\subsection{General setting}\label{s:general assumptions}

Throughout this article, $G$ denotes an lcsc group and $K\subset G$ a compact subgroup. We fix a left-Haar measure $m_G$ on $G$ and by $m_K$ the Haar measure on $K$, normalized such that $m_K(K)=1$. 
For $g\in G$ we let $L_g$ denote the left-multiplication by $G$ and $R_g$ the right-multiplication by $g^{-1}$.
For any function $f:G\to \C$ we denote by $\check{f}$ the function defined by $\check{f}(g) = f(g^{-1})$.
If $Y$ is any topological space, we will denote its Borel $\sigma$-algebra by $\mathcal{B}(Y)$.

In the following we fix a complete proper $G$-invariant metric $d$ on $X := G/K$, inducing the topology on $X$, and write $x_0\coloneqq eK\in X$. We denote by $\pi$ the quotient map $\pi:G\to G/K,\ g\mapsto gx_0 = gK$.
We set
\[
C_c(G, K)\coloneqq \{f\in C_c(G) : f(k_1gk_2) = f(g)\text{ for all }k_1,k_2\in K\},
\]
the set of all complex valued bi-$K$-invariant compactly supported continuous functions.
This is a $*$-algebra with multiplication given by convolution 
\[f_1*f_2(g) = \int f_1(h)f_2(h^{-1}g)dm_G(h)\]
and involution given by $f^*(g)\coloneqq \overline{f(g^{-1})}$. 
We denote by $L^1(G, K)$ its closure in $L^1(G)$. 

We say that $(G,K)$ is a \emph{Gelfand pair} if $C_c(G,K)$ (or equivalently $L^1(G, K)$) is commutative.

In the rest of this article we will always assume that $(G,K)$ is a Gelfand pair, if not specified otherwise, and use the notation described above. We will indicate further whenever we assume that $G$ is a Lie group.

If $(G, K)$ is a Gelfand pair, then $G$ is unimodular and thus there is a unique $G$-invariant Borel measure $m_X$ on $X\coloneqq G/K$ such that the Weil desintegration formula
\[\int_G f(g)dm_G(g) = \int_X\int_Kf(gk)dm_K(k)dm_X(gK)\]
holds for all $f\in C_c(G)$.
More generally, if $\mu$ is a $G$-invariant Borel measure on $X$, there exists some constant $C\geq0$ such that $\mu = Cm_X$.

For $f\in C_c(G)$ the function $f^\natural:G\to \C$, defined by
\[f^\natural(g)\coloneqq \int_K\int_K f(k_1gk_2)dm_G(g)\]
is in $C_c(G, K)$ and is called the \emph{$K$-periodization} of $f$.

If $f:G\to \C$ is a bi-$K$-invariant function on $G$, then there is a $K$-invariant function $f_K$ on $G/K$ given by $f_K(gK) = f(g)$.
If $h:G/K\to \C$ is a $K$-invariant function, then there is a bi-$K$-invariant function $h^K$ on $G$ given by $h^K(g) = h(gK).$
Note that $(f_K)^K=f$ and $(h^K)_K = h$. $f\mapsto f_K$ and $h\mapsto h^K$ send continuous functions to continuous functions, functions with compact support to functions with compact support and smooth functions to smooth functions (as the quotient map admits smooth local sections). 
For any $K$-invariant function $h:G/K\to \C$ such that $h^K\in L^1(G, K)$, we define $\widehat{h}=\widehat{h^K}$.	

\subsection{The spherical transform}
\subsubsection{The spherical transform of functions}
We begin by reviewing the spherical transform of functions on Gelfand pairs. The following is well known and can be found in the books \cite{Dieudonne}, \cite{VaradarajanGangolli}, \cite{vanDijk} and \cite{Wolf}. 

Denote by $C(G, K)$ the set of bi-$K$-invariant continuous functions on $G$. 
The convolution algebra $L^1(G, K)$ equipped with the $L^1$-norm is a commutative Banach algebra with involution. 
Each character $\phi:L^1(G, K)\to \C$ is of the form
\begin{equation}\label{e:character_spherical_function}
\phi(f)= \int f(g)\omega(g^{-1})dm_G(g),
\end{equation}
for some bounded $\omega\in C(G, K)$ satisfying
\begin{equation}\label{e:Godement_characterization_spherical_functions}
	\int_K \omega(g_1kg_2)dm_K(k) = \omega(g_1)\omega(g_2)
\end{equation}
for all $g_1,g_2\in G$.
If $\omega\in C(G, K)$ is bounded and satisfies equation \eqref{e:Godement_characterization_spherical_functions}, then it induces a character $\phi_\omega$ of the commutative Banach algebra $L^1(G, K)$ by formula \eqref{e:character_spherical_function} and is called a bounded spherical function.
We denote the set of bounded spherical functions of $(G, K)$ by $BS(G, K)$ define
\begin{equation*}\label{e:gelfand_transform_of_function}
G(f):BS(G, K)\to \C,\quad \omega \mapsto \phi_\omega(f)
\end{equation*}
for $f\in L^1(G, K)$.
Note that this is just the Gelfand transform of $f$, rewritten by parametrizing the Gelfand spectrum of $L^1(G, K)$ by $BS(G, K)$.
We equip $BS(G, K)$ with the weak topology induced by the family $\{G(f)\}_{f\in L^1(G, K)}$ of maps.
Then $BS(G,K)$ is a locally-compact Hausdorff space
and we have a map
\[
G:L^1(G,K)\to C_0(BS(G, K)), \quad f\mapsto G(f).
\]
We denote by $PS(G, K)$ the set of positive-definite functions in $BS(G, K)$, i.e. the set of functions $\phi\in BS(G, K)$
such that
\[\int\int \phi(h^{-1}g)\overline{f(h)}f(g)dm_G(h)dm_G(g)\geq 0\]
for all $f\in L^1(G)$.
\begin{definition}
	We define the \emph{spherical transform} of $f\in L^1(G, K)$ by  
	\[
	\widehat{f}\coloneqq G(f)|_{PS(G, K)}:PS(G,K)\to \C.
	\]
\end{definition}

Note that the spherical transform satisfies $\widehat{f^*} = \overline{\widehat{f}\hspace{3pt}}$ for all $f\in L^1(G, K)$.
\subsubsection{The spherical transform of measures}

\begin{definition}
	Let $\mu$ be a Borel measure on $G$. 
	\begin{enumerate}[(i)]
		\item $\mu$ is called \emph{positive-definite} if 
		\[\mu(f^**f)\geq 0\quad\text{for all $f\in C_c(G)$.}\]
		\item $\mu$ is called \emph{$K$-spherical} if $\mu(f) = \mu(f^\natural)$ for all $f\in C_c(G)$.
	\end{enumerate}
\end{definition}
Positive-definite measures are important because they admit spherical transforms by the Plancherel-Godement theorem:
\begin{theorem}[Godement, \cite{Godement}]\label{t:Godement-Plancherel_for_measures}
	Let $\mu$ be a positive-definite Borel measure on $G$. Then there is a unique regular Borel measure $\widehat{\mu}$ on $PS(G, K)$ such that
	\begin{enumerate}[(i)]
		\item $\widehat{f}\in L^2(\widehat{\mu})$ for all $f\in C_c(G,K)$,
		\item $\mu(f*g^*) = \widehat{\mu}(\widehat{f}\cdot\overline{\widehat{g}}\hspace{2pt})$ for all $f,g\in C_c(G, K)$.
	\end{enumerate}
\end{theorem}
A detailed proof of the Plancherel-Godement theorem can be found in \cite{Dieudonne}. A simple proof of a version for distributions was given by Barker in \cite{Barker} and can be modified to yield this version for measures. Motivated by the theorem above, we will write $C_c(G,K)^2$ for the complex linear span of $\{f*g\mid f,g\in C_c(G, K)\}$.
\subsubsection{The spherical transform of distributions} 
In order to prove obtain estimates involving Schwartz functions, we will take advantage of the Lie group structure of our Gelfand pairs by using a version of the spherical transform for distributions.

For a manifold $M$ let $D(M)$ denote the algebra of differential operators on $M$, i.e. the algebra generated by the derivations of $C^\infty(M)$ and the maps $D_f:C^\infty(M)\to C^\infty(M),\ g\mapsto fg$ for $f\in C^\infty(M)$.
\begin{definition}\label{d:distributions}
	Let $M$ be a manifold. A \emph{distribution} $T$ on $M$ is a linear map $T:C_c^\infty(M)\to \C$ such that for every open, relatively-compact $O\subset M$ there are finitely many $D_1,\dots, D_k\in D(M)$ such that
	\[\abs{T[f]}\leq \sum_{i=1}^k \norm{D_if}_\infty\quad\text{for all $f\in C_c^\infty(M)$ with $\mathrm{supp}(f)\subset O$.}\]
\end{definition}
A more detailed discussion of distributions on manifolds and Lie groups can be found in \cite{Helgason2} and in the appendix of \cite{WarnerSemmisimple1}.
\begin{lemma}\label{l:lf_measure_are_distributions}
	Assume that $G$ is a Lie group and let $\mu$ be a Randon measure on $G$. Then the map 
	\[C_c^\infty(G)\to \C,\quad f\mapsto \mu(f)\]
	is a distribution on $G$.
\end{lemma}
\begin{proof}
	Let $C\subset M$ be open and relatively-compact and let $f\in C_c^\infty(M)$ with $\mathrm{supp}(f)\subset C$. Then
	\[\abs{\int fd\mu}\leq \mu(C)\norm{f}_\infty\leq \mu(\overline{C})\norm{f}_\infty.\]
\end{proof}
\begin{definition}
	Assume that $G$ is a Lie group. 
	\begin{enumerate}[(i)]
		\item A distribution $T$ on $G$ is called \emph{positive-definite}, if 
		\[T[f^**f]\geq 0\quad\text{ for all $f\in C_c^\infty(G)$.}\]
		\item A distribution $T$ on $G$ is called $\emph{$K$-spherical}$, if
		$T[f] = T[f\circ L_k] = T[f\circ R_k]$ for all $k\in K$ and $f\in C_c^\infty(G)$.
	\end{enumerate}
\end{definition}
\begin{theorem}[Godement, \cite{Godement}]\label{t:Godement-Plancherel_for_distributions}
	Assume that $G$ is a Lie group. Let $T$ be a positive-definite distribution on $G$. Then there is a unique regular Borel measure $\widehat{T}$ on $PS(G, K)$ such that
	\begin{enumerate}[(i)]
		\item $\widehat{f}\in L^2(\hspace{2pt}\widehat{T}\hspace{2pt})$ for all $f\in C_c^\infty(G,K)$,
		\item $T[f*g^*] = \widehat{T}(\widehat{f}\cdot\overline{\widehat{g}}\hspace{2pt})$ for all $f,g\in C_c^\infty(G, K)$.
	\end{enumerate}
\end{theorem}
A proof of the theorem can be found in \cite{BarkerThesis}.
An important special case occurs if the distribution is given by integration against a positive-definite Borel measure.
The uniqueness in the Plancherel-Godement theorems for measures and distributions now implies the following:
\begin{corollary}\label{c:distributional_plancherel-godement_is_measure_plancherel_godement}
	Assume that $G$ is a Lie group and $\mu$ is a positive-definite Radon measure. Then  $T_\mu=\mu|_{C_c^\infty(G)}$ is a distribution and
	\[\widehat{T_\mu} = \widehat{\mu}.\]
\end{corollary}
\begin{remark}
	In all of our examples this distribution extends to a tempered distribution by appropriate versions of the Bochner-Schwartz theorem. This tempered distribution is no longer necessarily given by integration against $\mu$.
\end{remark}
\subsection{Examples}
\subsubsection{Euclidean space}
See \cite{vanDijk}, \cite{Wolf} or \cite{AnkerDunklTheory} for the material of this section.
\begin{enumerate}[(i)]
	\item Consider the pair $G=\R^n$, $K=\{0\}$. Then $X = G/K = \R^n$ and as convolution of functions in $L^1(G, K) = L^1(\R^n)$ is commutative, $(\R^n, \{0\})$ is a Gelfand pair. For each $\xi\in \R^n$, let 
	\[\phi_\xi:\R^n\to \C,\quad x\mapsto \exp(2\pi i\langle x, \xi\rangle).\]
	Then $PS(\R^n, \{0\}) = \{\phi_\xi\mid \xi\in \R^n\}$ and the spherical transform is the ordinary Fourier transform. 
	Let $\mathscr{S}(\R^n)$ denote the Schwartz space on $\R^n$.
	\begin{theorem}\label{t:euclidean_Bochner_Schwartz}
		If $T$ is a positive-definite distribution on $\R^n$, then $T$ extends uniquely to a tempered distribution $\tilde{T}$ and for all $f\in \mathscr{S}(\R^n)$ we have
		\[\tilde{T}f = \widehat{T}(\widehat{f}\hspace{1pt}).\]
	\end{theorem}
	\item Consider the pair $G=\R^n\rtimes SO(n)$, $K=SO(n)$. Then $X=G/K=\R^n$ and one can show that $(\R^n\rtimes SO(n), SO(n))$ is a Gelfand pair. 
	If $f$ is a bi-$K$-invariant function on $\R^n\rtimes SO(n)$, then $f$ satisfies
	\[f((0,A)(x, B)(0, (AB)^{-1}) = f((Ax, I_n))\]
	for all $x\in \R^n$ and $A,B\in SO(n)$. Thus there is a function $f_0:[0,\infty)\to \C$ with $f(x)= f_0(\norm{x})$.
	
	The spherical transform of $f$ is related to the Hankel transform of $f_0$.
	More specifically, the spherical functions of $(\R^n\rtimes SO(n), SO(n))$ are given by
	\[\varphi_\lambda(x, A)\coloneqq \Gamma\left(\frac{n}{2}\right)\left(\frac{\lambda\norm{x}}{2}\right)^{(2-n)/2}J_{(n-2)/2}(\lambda\norm{x})\qquad\text{ for }\lambda\geq 0,\]
	where $J_{k}$ denotes the Bessel function of the first kind of order $k\geq 0$. 
	Thus the spherical transform of $f$ is given by
	\begin{align*}
		\widehat{f}(\lambda)&\coloneqq \int_{\R^n}\int_{SO(n)} f(x,A)\varphi_\lambda(-A^{-1}x,A^{-1})dm_{SO(n)}(A)dx\\
		&= (2\pi)^{n/2}\frac{1}{\lambda^{(n-2)/2}}H_f((n-2)/2,\lambda),
	\end{align*}
	where $H_f$ denotes the Hankel transform of $r\mapsto r^{-n/2}f_0(r)$. Note that this is just the Euclidean Fourier transform of $f$ in radial coordinates.
\end{enumerate}

\subsubsection{Heisenberg space}
Set $\mathcal{H}_n\coloneqq \R\times \C^n$ and for $(t,x), (s, y)\in \mathcal{H}_n$ define
\[(t,v)\cdot (s,w)\coloneqq (ts-\frac{1}{2}\mathrm{Im} \langle v,w\rangle, v+w).\]
Then $(\mathcal{H}_n,\cdot)$ is a nilpotent Lie group, called \emph{Heisenberg group}.
The group $U(n)$ acts via
\[k.(t,v)\coloneqq (t, kv)\quad\text{for $k\in U(n)$}\]
by automorphisms on $\mathcal{H}_n$. Thus we can form the semidirect product $\mathcal{H}_n\rtimes U(n)$, called the \emph{Heisenberg motion group}. We note that $\mathcal{H}_n=\mathcal{H}_n\rtimes U(n)/U(n)$, which is why we will also call $\mathcal{H}_n$ the \emph{Heisenberg space}.
We can equip $\mathcal{H}_n$ with the Cygan-Koranyi metric $d_{CK}$, which is induced by the group norm
\[
\norm{(t,v)}\coloneqq (t^2+\norm{v}_2^4)^{1/2}
\]
on $\mathcal{H}_n$ and is proper, left-invariant and complete.\footnote{There are many other proper, left-invariant and complete metrics on $\mathcal{H}_n$.} We note that the Heisenberg group can be equipped with a family $(D_r)_{r> 0}$ of group automorphisms, defined by 
\[D_r:\mathcal{H}_n\to \mathcal{H}_n,\ (t,v)\mapsto (r^2t,rv).\]
These automorphisms are called \emph{dilations} of $\mathcal{H}_n$ and the Cygan-Koranyi metric is compatible with the family $(D_r)_{r>0}$
in the sense that
\[\norm{D_r(g)} = r\norm{g}\quad \text{for all }g\in \mathcal{H}_n, r>0.\]
This implies that 
\[D_r(B(g,s)) = B(D_r(g), rs)\]
 for all $g\in \mathcal{H}_n$ and $r,s>0$.
 For all measurable functions $f,\varphi:\mathcal{H}_n\to \C$ and all $r>0$ we have
 \[\int (f\circ D_r)(x) \varphi(x) dm_{\mathcal{H}_n}(x) = \frac{1}{r^{2n+2}}\int f(x)(\varphi\circ D_{1/r})(x)dm_{\mathcal{H}_n}(x)\]
 if the integrals exist.
 
The pair $(\mathcal{H}_n\rtimes U(n), U(n))$ forms a Gelfand pair. See \cite{Thangavelu} or \cite{Wolf} for a detailed exposition of the theory of spherical harmonic analysis of $(\mathcal{H}_n\rtimes U(n), U(n))$. 

\begin{theorem}[Benson--Jenkins--Ratcliff, \cite{BensonJenkinsRatcliff}]\label{t:parametrization_spherical_heisenberg}
	The spherical functions of $(\mathcal{H}_n\rtimes U(n), U(n))$ fall into the following two families:
	\begin{enumerate}[(A)]
		\item \[\phi_{\lambda, m}(t, v, k) = e^{i\lambda t}L_m^{(n-1)}\left (\frac{1}{2}\lambda\norm{v}_2^2\right )e^{-\frac{1}{4}\lambda\norm{v}_2^2}\] for $\lambda>0$ and $m\in \Z_+$ and 
		$\phi_{\lambda, m} = \overline{\phi_{\abs{\lambda}, m}}$ for $m\in \Z_+$ and $\lambda < 0$.
		\item \[\eta_\tau(t, v, k) = \frac{2^{n-1}(n-1)!}{(\tau\norm{v}_2)^{n-1}}J_{n-1}(\tau\norm{v}_2)\] for $\tau>0$ and $\eta_0(t,v,k)=1$ for all $(t,v,k)\in \mathcal{H}_n\rtimes U(n)$.
	\end{enumerate}
	Here 
	\[
	L_m^{(n-1)}(x) = (n-1)!\sum_{j=0}^m\binom{m}{j}\frac{(-x)^j}{(j+n-1)!}
	\]
	is the generalized Laguerre polynomial of order $n-1$ normalized to $1$ at $0$ and $J_{n-1}$ is the Bessel function of order $n-1$.
\end{theorem}
As the Haar measure $m_{\mathcal{H}_n}$ is given by $m_{\R}\otimes m_{\C^n}$, the spherical transform on $\mathcal{H}_n$ is just given by ordinary integration with respect to the Lebesgue measure on $\R\times \C^n$.

Denote the Lie algebra of the Heisenberg group $\mathcal{H}_n$ by $\mathfrak{h}_n$ and note that $\exp:\mathfrak{h}_n\to \mathcal{H}_n$ is a diffeomorphism. We define the Schwartz space $\mathscr{S}(\mathcal{H}_n)$ by 
\[\mathscr{S}(\mathcal{H}_n)\coloneqq \{f\circ \exp^{-1}\mid f\in \mathscr{S}(\mathfrak{h}_n)\}\]
and set
\[\mathscr{S}(\mathcal{H}_n, U(n))\coloneqq \{f\in \mathscr{S}(\mathcal{H}_n)\mid f(t,kv) = f(t,v)\text{ for all $k\in U(n)$}\}.\]

We define the space of bi-$U(n)$-invariant Schwartz functions on $\mathcal{H}_n\rtimes U(n)$ by
\[\mathscr{S}(\mathcal{H}_n\rtimes U(n), U(n)) = \{f\circ \pi\mid f\in \mathscr{S}(\mathcal{H}, U(n))\}\]
and topologize it such that the canonical map $\mathscr{S}(\mathcal{H}_n, U(n))\to \mathscr{S}(\mathcal{H}_n\rtimes U(n), U(n))$ is a topological isomorphism.

Due to lack of reference we prove the following theorem in Appendix \ref{S:Bochner-Schwarz appendix}:
\begin{theorem}\label{t:heisenberg_bochner_schwartz}
	Let $T$ be a positive-definite distribution on $\mathcal{H}_n\rtimes U(n)$. Then there is a (unique) continuous functional
	\[\tilde{T}:\mathscr{S}(\mathcal{H}_n\rtimes U(n), U(n))\to \C\]
	such that 
	\[\tilde{T}f = Tf\]
	for any $f\in C^\infty_c(\mathcal{H}_n\rtimes U(n),U(n))$
	and 
	\[\tilde{T}f = \widehat{T}(\widehat{f}\hspace{1pt})\]
	for all $f\in \mathscr{S}(\mathcal{H}_n\rtimes U(n), U(n))$.
\end{theorem}
\subsubsection{Noncompact semisimple Lie groups}\label{ss:noncompt_ss_Lie_harmonic}
Assume now that $G$ is a (connected) semisimple Lie group with finite center and no compact factors and choose a Cartan decomposition $\mathfrak{g} = \mathfrak{k}\oplus\mathfrak{p}$ of the Lie algebra of $G$. 
The following is well-known (see \cite{VaradarajanGangolli} and \cite{Knapp}) and we use the usual notation for this setup, i.e.
\begin{enumerate}[(i)]
	\item $\mathfrak{a}$ is a maximal abelian subspace of $\mathfrak{p}$,
	\item $\rho$ is the Weyl vector associated to the restricted root space decomposition with respect to $\mathfrak{a}$,
	\item $K$ is the maximal compact subgroup of $G$ associated to $K$,
	\item $\kappa$ is the Killing form of $\mathfrak{g}$,
	\item $W$ is the Weil group of the restricted root system.
\end{enumerate} 
$(G, K)$ is a Gelfand pair and $G$ has the Cartan decomposition $G=K\exp(\mathfrak{p})$. The Iwasawa decomposition $G=KAN$ allows us to define \[H:G=KAN\to \mathfrak{a},\ g=kan\mapsto \exp^{-1}(a).\]
We will assume that the measures $m_K, m_A, m_N$ and $m_\mathfrak{a}$ obey the standard normalization, see Section 2.4 of \cite{VaradarajanGangolli}. 
Using the function $H$, for each $\lambda\in \mathfrak{a}_\C^*$, the dual of the complexification of $\mathfrak{a}$, we define the function
\[\varphi_\lambda(g)=\int_Ke^{(i\lambda-\rho)(H(gk))}dm_K(k).\]
This formula parametrizes a superset of all bounded spherical functions of $(G,K)$ and is called the Harish-Chandra parameterization (see \cite{Helgason2} and \cite{VaradarajanGangolli}). It has the following properties:
\begin{enumerate}[(i)]
	\item $\varphi_\lambda = \varphi_\mu$ iff there is some $w\in W$ with $w\lambda = \mu$.
	\item $\varphi_{-\lambda}(g) = \varphi_\lambda(g^{-1})$.
	\item If $\varphi_\lambda\in PS(G, K)$, then there is a $w\in W$ with $w\lambda = \overline{\lambda}$.
\end{enumerate}
Define the tube domains $\mathscr{F}^\varepsilon\coloneqq \mathfrak{a}^*+i\varepsilon C$, where $C$ is the closed convex hull of $\{w\rho\mid w\in W\}$.
The Helgason-Johnson theorem states that
\[BS(G, K) = \{\varphi_\lambda\mid \lambda\in \mathscr{F}^1\}\]
and using property (i) above, one can identify $BS(G, K)$ with the set $\mathscr{F}^1/W$.

The Harish-Chandra $\Xi$-function is defined as 
\[\Xi\coloneqq \varphi_0\]
and one defines 
\[\sigma:G=K\exp(\mathfrak{p})\to \R,\ k\exp(X)\mapsto \sqrt{\kappa(X,X)}.\]
Using $\Xi$ and $\sigma$ one can define the Harish-Chandra $L^p$-Schwartz seminorms
\[q_{D,E, m, p}(f)\coloneqq \sup_{g\in G} \frac{ \abs{f(D;g;E)}}{(1+\sigma(g))^{-m}\Xi(g)^{-2/p}},\] where $0< p\leq 2$, $D,E\in \mathcal{U(\mathfrak{g})}$ and $m\in \N_0$. 

Here $\mathcal{U}(\mathfrak{g})$ denotes the universal enveloping algebra of $\mathfrak{g}$ and $f(D;\cdot;E)$ is the function obtained from $f$ by acting on the left by the differential operator $D$ and from the right by the differential operator $E$.

We define the Harish-Chandra $L^p$-Schwartz spaces for $0<p\leq 2$ as
\[
\mathscr{S}^p(G)\coloneqq \left\{f\in C^\infty(G)\setdiv \forall m\in \N_0\forall D,E\in \mathcal{U}(\mathfrak{g}) : q_{D,E,m,p}(f)<\infty \right\}.
\]
The spaces $\mathscr{S}^p(G)$ are topologized by the families 
\[
(q_{D,E,m,p})_{D,E\in \mathcal{U(\mathfrak{g})},m\in \N_0}
\]
 of seminorms and we denote by $\mathscr{S}^p(G, K)$ the set of bi-$K$-invariant functions in $\mathscr{S}^p(G)$.
It can be shown that $\mathscr{S}^p(G,K)\subset \mathscr{S}^q(G, K)$ for $p\leq q$.

The Harish-Chandra transform of $f\in \mathscr{S}^2(G, K)$ is defined by
\[\mathcal{H}(f)(\lambda) \coloneqq \int f(g)\varphi_{-\lambda}(g)dm_G(g)\]
and the Abel transform of $f$ is defined by
\[\mathcal{A}(f)(H) = e^{\rho(H)}\int_Nf(\exp(H)n)dm_N(n).\]
These transforms fit in the following commutative diagram of isomorphism:
\begin{center}
	\begin{tikzcd}
		{\mathscr{S}^2(G,K)} \arrow[rr, "\mathcal{H}"] \arrow[rd, "\mathcal{A}"'] &   & \mathscr{S}(\mathscr{F}^0)^W \\
		& \mathscr{S}(\mathfrak{a})^W \arrow[ru, "\mathcal{F}"'] &                             
	\end{tikzcd}
\end{center}

Here $\mathcal{F}$ denotes the Fourier transform (note that any factors of $2\pi$ disappear by the standard normalization of the Haar measures)
\[\mathcal{F}(f)(\lambda)\coloneqq \int_\mathfrak{a} f(H)e^{-i\lambda(H)}dm_\mathfrak{a}(H)\]
and $\mathscr{S}(\mathfrak{a})^W$ and $\mathscr{S}(\mathscr{F}^0)^W$ denote the Weyl group invariant elements of the ordinary Schwartz spaces on the real vector spaces $\mathfrak{a}$ and $\mathscr{F}^0=\mathfrak{a}^*$.
A theorem by Trombi and Varadarajan characterizes the image of $\mathscr{S}^P(G, K)$ under $\mathcal{H}$  in terms of spaces of functions on tube domains, see for instance Theorem 7.10.9 in \cite{VaradarajanGangolli} for a precise statement.

A continuous functional $\mathscr{S}^p(G)\to \C$ is called a $L^p$-tempered distribution.
\begin{theorem}[Spherical Bochner-Schwartz theorem (Barker,\cite{Barker})]\label{t:noncompact_ss_bochner_schwartz}
	In the setting above assume that $T$ is a positive-definite distribution on $G$. Then $T$ has a unique extension to an $L^1$-tempered distribution $\tilde{T}$ and for all $f\in \mathscr{S}^1(G, K)$ the formula
	\[\tilde{T}f = \int \widehat{f}d\widehat{T}\]
	holds.
\end{theorem}

The spaces $G/K$ are equipped with a natural left-invariant Riemannian metrics induced by the bilinear form $\kappa$ on $\mathfrak p \cong T_{x_0}(G/K)$. Note that $\kappa|_{\mathfrak{p}\times \mathfrak{p}}$ is positive-definite. A wealth of information about the geometry and harmonic analysis on these spaces can be found in the books \cite{Helgason} and \cite{Helgason2} by Helgason.
\begin{remark}\label{r:harmonic_analysis_hyperbolic_space}
	A particularly interesting case occurs in this family of examples when $G=SO(n,1)$ and $K=SO(n)$, $n\geq 2$.
	
	In this case $G/K$ with the Riemannian metric induced by $\kappa$ is isometric to the hyperbolic $n$-space $\mathbb{H}^n$.
	Moreover $A$ is one-dimensional, such that bi-$K$-invariant functions only depend on $r=d(gx_0,x_0)$.
	
	Thus the value of a spherical function $\varphi$ at $g\in G$ only depends on the distance $r=d(gx_0, x_0)$. 
	Set $\rho = \frac{n-1}{2}$. 
	Then the positive definite spherical functions in radial coordinates are given by
	\begin{align*}
		\varphi_\lambda(r)= \varphi_\lambda^{(\rho, -\frac{1}{2})}(r) = ~_2F_1\left(\frac{\rho+i\lambda}{2}, \frac{\rho-i\lambda}{2};\frac{n}{2}; -\sinh(r)^2\right)
	\end{align*}
	with \[\lambda\in i[0,\rho] \cup [0,\infty),\]
	where $~_2F_1(a,b;c;z)$ denotes the Gauss hypergeometric function and the $\varphi_\lambda^{(\alpha,\beta)}$ denote the Jacobi functions.
	The spherical transform in radial coordinates is then given by
	\begin{align*}
		\widehat{f}(\lambda) = 2\frac{\pi^{n/2}}{\Gamma(n/2)}\int_0^\infty f(r)\varphi_{\lambda}(r)\sinh(r)^{n-1}dr,
	\end{align*}
	with the additional $\sinh$-term coming from the Haar measure.
	
	In this case there is also an explicit formula for the Abel transform, given in radial coordinates by
	\[\mathcal{A}f(r) = \frac{(2\pi)^{\frac{n-1}{2}}}{\Gamma(\frac{n-1}{2})}\int_{\abs{r}}^\infty\sinh(s)(\cosh(s)-\cosh(r))^{\frac{n-3}{2}}f(s)ds.\]
	
	In odd dimensions the inverse of the Abel transform can be obtained by
	\[\mathcal{A}^{-1}(f)(r) = (2\pi)^{-\frac{n-1}{2}}\left(-\frac{1}{\sinh(r)}\frac{\partial}{\partial r}\right)^{\frac{n-1}{2}}f(r)\]
	and in even dimensions by 
	\[\mathcal{A}^{-1}(f)(r) = \frac{1}{2^{\frac{n-1}{2}}\pi^{\frac{n}{2}}}\int_{\abs{r}}^\infty\frac{-\frac{\partial}{\partial s}\left (-\frac{1}{\sinh(r)}\frac{\partial}{\partial s}\right)^{n/2-1}g(s)}{\sqrt{\cosh(s)-\cosh(r)}}ds.\]

	
	For $G = SO(n,1)$, $K=SO(n)$ the spaces $\mathscr{S}^p(G, K)$ can be identified via radial coordinates with the spaces \[\cosh^{-\frac{n-1}{p}}\mathscr{S}_{\mathrm{even}}(\R),\] where $\mathscr{S}_\mathrm{even}(\R)$ denotes the space of even Schwartz functions on $\R$ (see Theorem 6.1 in \cite{Koornwinder} or 2.28 in \cite{AnkerHarmonicANGroups}).
	
	The image of $\mathscr{S}^p(G,K)$, $1\leq p \leq 2$, under the Harish-Chandra transform $\mathcal{H}$ is given by the set $\overline{\mathscr{L}}(\mathscr{F}^{1/p-1/2})$ of smooth even functions $f$ on the strip
	\[\mathscr{F}^{1/p-1/2}=\left\{\lambda\in \C\ \middle|\ \abs{\mathrm{Im}(\lambda)}\leq \left(\frac{1}{p}-\frac{1}{2}\right)(n-1)\right\},\]
	which are holomorphic on its interior and satisfy
	\[\sup_{\abs{\mathrm{Im}(\lambda)}\leq \left(\frac{1}{p}-\frac{1}{2}\right)(n-1)}(1+\abs{\lambda})^N\left\lvert\left(\frac{d}{d\lambda}\right)^Mf(\lambda)\right \rvert<\infty\]
	for all $M,N\in \N_0$. The image of $\mathscr{S}^p(G,K)$ under the Abel transform can be identified with the space
	\[\cosh^{-(1/p-1/2)(n-1)}\mathscr{S}_{\mathrm{even}}(\R)\]
	and we have the following commutative diagram of isomorphisms:
	\begin{center}
	\begin{tikzcd}
		\cosh^{-\frac{n-1}{p}}\mathscr{S}_{\mathrm{even}}(\mathbb R) \arrow[rr, "\mathcal{H}"] \arrow[rd, "\mathcal{A}"'] &                                                                                                      & \overline{\mathscr{L}}(\mathscr{F}^{1/p-1/2}) \\
		& \cosh^{-(1/p-1/2)(n-1)}\mathscr{S}_\mathrm{even}(\mathbb R) \arrow[ru, "\mathcal{F}"'] &                                              
	\end{tikzcd}
	\end{center}
	See \cite{Koornwinder} and \cite{AnkerHarmonicANGroups} for further details and references.
	
	For semisimple groups with $\dim(A) = 1$, similar formulas for the spherical functions hold, see for instance \cite{Koornwinder} or \cite{Wolf}. Moreover an explicit inversion formula for the Abel transform is known, see \cite{Koornwinder} and \cite{AnkerDunklTheory}.
	
	In the case of $\dim(A)>1$, one must replace the hypergeometric functions by multivariable analogues. These can be handled in a somewhat explicit way with Dunkl theory, see \cite{AnkerDunklTheory} and the references therein. 
\end{remark}
\subsubsection{Compact semisimple Lie groups}
Let $(G, K)$ be a Riemannian symmetric pair with $G$ compact and semisimple (and connected). We denote by $\mathfrak{g}$ the Lie algebra of $G$, by $\mathfrak{k}$ the Lie algebra of $K$ and by $\theta$ the Cartan involution with $G^\theta_0\subset K\subset G^\theta$, where $G^\theta$ is the group of fixed points of $G$ and $G^\theta_0$ its connected component. 

Let $\kappa$ denote the Killing form on $\mathfrak{g}$. Then $-\kappa$ defines a scalar product on $\mathfrak{g}$ which extents uniquely to a complex scalar product $\langle \cdot,\cdot\rangle$ on $\mathfrak{g}_\C\coloneqq \mathfrak{g}\otimes \C$.

Let $\mathfrak{g} = \mathfrak{k}\oplus\mathfrak{p}$ denote the Cartan decomposition with respect to $\theta$ and choose a maximal abelian subspace $\mathfrak{a}\subset \mathfrak{p}$. Let $\Sigma\subset i\mathfrak{a}^*$ denote the set of restricted roots of $\mathfrak{g}_\C$ with respect to $\mathfrak{a}_\C$ and $\Sigma^+$ a choice of positive roots.

Denote the universal covering group of $G$ by $\tilde{G}$. The involution $\theta$ lifts to an involution $\tilde{\theta}$ on $\tilde{G}$ and $\tilde{K} = \tilde{G}^{\tilde{\theta}}$ is connected. 
\begin{theorem}[Helgason, \cite{Helgason2}]
	$PS(\tilde{G}, \tilde{K})$ is in bijection with the set
	\[
	\Lambda^+(\tilde{G}/\tilde{K})\coloneqq \left\{\lambda\in i\mathfrak{a}^*\mid \forall \alpha\in \Sigma^+:\frac{\langle \lambda,\alpha\rangle}{\langle \alpha,\alpha\rangle}\in \Z^+\right\}.
	\]
	More precisely, there is a bijection between $\Lambda^+(\tilde{G}/\tilde{K})$ and the set of equivalence classes of irreducible $\tilde{K}$-spherical representations sending $\lambda\in \Lambda^+(\tilde{G}/\tilde{K})$ to the equivalence class of irreducible representations with highest weight $\lambda$.  
\end{theorem}
For $\lambda\in \Lambda^+(\tilde{G}/\tilde(K))$ let $(\pi_\lambda,V_\lambda)$ denote a fixed irreducible representation with highest weight $\lambda$.
Let $\Lambda^+(G/K)$ denote the set of $\lambda\in \Lambda^+(\tilde{G}/\tilde{K})$ such that $(\pi_{\lambda}, V_\lambda)$ descends to a $K$-spherical representation of $G$. 
\begin{lemma}[Olafsson--Schlichtkrull, \cite{OlafssonSchlichtkrull}]
	There are $\lambda_1,\dots,\lambda_k\in \Lambda^+(\tilde{G}/\tilde{K})$ such that
	 \[\Lambda^+(G/K) = \Z^+\lambda_1\oplus \dots\oplus \Z^+\lambda_k\subset i\mathfrak{a}^*\] with $k=\dim\mathfrak{a}$.
\end{lemma}
For $\lambda\in \Lambda^+(G/K)$ we denote by $\phi_\lambda$ the spherical function associated to $(\pi_\lambda, V_\lambda)$.
A $K$-invariant distribution on $X=G/K$ is an continuous functional $T$ on $C^\infty(X)$ such that
\[T(f) = T(f(k\cdot))\quad\text{for all }f\in C^\infty(G/K) \text{ and } k\in G.\] 
\begin{theorem}[Olafsson--Schlichtkrull,\cite{OlafssonSchlichtkrull}]
	Let $T$ be a $K$-invariant distribution on $X=G/K$. Then for any $f\in C^\infty(G, K)$ we have
	\[T(f_K) = \sum_{\lambda\in \Lambda^+(G/K)}\dim(V_\lambda)\widehat{f}(\phi_\lambda)T(\phi_{\lambda^*}),\]
	where $\phi_{\lambda^*}$ is defined by $\phi_{\lambda*}(g) = \phi_\lambda(g^{-1})$.
\end{theorem}
\begin{corollary}\label{c:compact_ss_Bochner_Schwartz}
	Let $T$ be a bi-$K$-invariant distribution on $G$. Then 
	\[T(f) = \widehat{T}(\widehat{f})\]
	for all $f\in C^\infty(G, K)$.
\end{corollary}
\begin{remark}
	A particularly interesting case occurs in this family of examples when $G=SO(n+1)$ and $K=SO(n)$. In this case $G/K$ can be identified with the $n$-sphere $\mathbb{S}^n=\partial B(0,1)\subset \R^{n+1}$ as $SO(n+1)$ acts transitively on $\mathbb{S}^n$ and $e_1=(1,0,\dots, 0)^T$ is stabilized by a subgroup isomorphic to $SO(n)$. 
	
	We will quickly review the Haar measure and spherical functions for the pair $(SO(n+1), SO(n))$, see \cite{AnkerDunklTheory} and references therein for more information.
	
	Let $\theta_1$ denote the angle between $x$ and $e_1\coloneqq (1,0,\dots, 0)^\top$.
	We can identify bi-$SO(n)$-invariant functions with $SO(n)$-invariant functions on $\mathbb{S}^n$.
	With respect to the usual spherical coordinates these functions only depend on $\cos(\theta_1)$.
	Integrals of radial measurable functions with respect to the Haar measure are then given by
	\[\int f(g)dm_{SO(n+1)}(g) = 2\frac{\pi^{n/2}}{\Gamma(n/2)}\int_0^\pi f_{SO(n)}(\cos(\theta_1))\sin(\theta_1)^{n-1}d\theta_1\]
	The spherical functions on $\mathbb{S}^n$ are of the form
	\begin{align*}
	\phi_\lambda(\cos(\theta_1)) = \frac{\lambda!(n-2)!}{(\lambda+n-2)!}C_\lambda^{(\frac{n-2}{2})}(\cos(\theta_1))	=\frac{\lambda!}{(\frac{n}{2})_\lambda}P_\lambda^{(\frac{n}{2}-1,\frac{n}{2}-1)}(\cos(\theta_1))
	\end{align*}
	where $\lambda\in \mathbb{N}$. Here the $C_k^{(m)}$ denote the Gegenbauer polynomials, the $P_\lambda^{(\alpha,\beta)}$ denote the Jacobi polynomials and $(m)_k$ denotes the falling Pochhammer symbol. As in the case of the spherical functions for hyperbolic space they can also be expressed in terms of the Gauss hypergeometric function:
	\[\phi_\lambda(\cos(\theta_1)) = ~_2F_1(-\lambda, \lambda+n-1;n/2;\sin(\theta_1)^2).\] 
\end{remark}

\subsection{Schwartz-like function spaces}
Assume in this subsection that $G$ is a Lie group.

\begin{definition}\label{d:Schwartz-like}
	We say that a topological $*$-subalgebra $\mathcal{S}(G,K)\subset L^1(G,K)\cap C(G)$ (with a possibly finer topology) containing $C_c^\infty(G,K)$ is \emph{Schwartz-like}, if 
	\begin{enumerate}[(i)]
		\item for every $f\in \mathcal{S}(G, K)$ without compact support there is a sequence $(g_n)_{n\geq 1}$ in $C_c^\infty(G,K)$ such that \begin{enumerate}[a)]
			\item $g_nf\in C_c^\infty(G,K)$ for all $n\in \N$,
			\item $g_nf\to f$ in $\mathcal{S}(G,K)$,
			\item $g_n$ takes values in $[0,1]$,
			\item $g_n(e)=1$.
		\end{enumerate}
		\item $\widehat{f}\in L^1(\widehat{\mu})$ for all $f\in \mathcal{S}(G,K)$ and spherical positive-definite Radon measures $\mu$ on $G$.
		\item For any positive-definite Radon measure $\mu$ there is a unique continuous linear functional $T_\mu:\mathcal{S}(G,K)\to \C$ such that
		\[T_\mu(f) = \int \widehat{f}d\widehat{\mu}\]
		for all $f\in \mathcal{S}(G,K)$ and $T_\mu|_{C_c^\infty(G,K)} = \mu|_{C_c^\infty(G,K)}$.
	\end{enumerate} 
\end{definition}
\begin{proposition}\label{e:examples_Schwartz_like}
	The following algebras are Schwartz-like:
	\begin{enumerate}[(i)]
		\item $C_c(G,K)^2=\lspan\{f*g\mid f,g\in C_c(G,K)\}$ for a general Lie group Gelfand pair.
		\item $\mathscr{S}(\R^n)$.
		\item $\mathscr{S}^1(G,K)$, where $G$ is a semisimple Lie group with finite center and no compact factors, $K$ a maximal compact subgroup.
		\item $C^\infty(G,K)$ for Riemannian symmetric pairs $(G,K)$ of compact type.
		\item $\mathscr{S}(\mathcal{H}_n\rtimes U(n), U(n))$.
	\end{enumerate}
\end{proposition}
\begin{proof}[Proof]
	\begin{enumerate}[(i)]
		\item As every function in $C_c(G, K)^2$ has compact support, Condition (i) holds vacuously. Conditions (ii) and (iii) follow directly from the Godement-Plancherel theorem \ref{t:Godement-Plancherel_for_measures}.
		\item This follows from the Bochner-Schwartz theorem together with Corollary \ref{c:distributional_plancherel-godement_is_measure_plancherel_godement} and \cite[Theorem 1.8.7]{GrafakosTextbook}.
		\item This follows from Barkers spherical Bochner-Schwartz theorem \ref{t:noncompact_ss_bochner_schwartz} together with Corollary \ref{c:distributional_plancherel-godement_is_measure_plancherel_godement} and the remarks after \cite[Definition 7.8.4]{VaradarajanGangolli} together with \cite[Lemma 6.1.7]{VaradarajanGangolli}.
		\item This follows directly from Corollary \ref{c:compact_ss_Bochner_Schwartz}.
		\item The existence of a sequence $(g_n)_{n\geq 1}$ follows directly from the Euclidean case, by the definition of Schwartz space. Now Theorem \ref{t:heisenberg_bochner_schwartz} together with Corollary \ref{c:distributional_plancherel-godement_is_measure_plancherel_godement} implies that $\mathscr{S}(\mathcal{H}_n\rtimes U(n), U(n))$ is Schwartz-like.\qedhere
	\end{enumerate}
\end{proof}

\section{Point processes in proper commutative spaces}\label{s:PointProcesses}
In this section we first recall the notion of point processes. For this we introduce a space measures and measure valued random variables. We then define the intensity and autocorrelation of point processes.
\subsection{Point processes}
We will now give a quick exposition of point processes theory in $X$ and use the opportunity to fix our notation for point process theory.
Assume that $(G,K)$ is a Lie group Gelfand pair and that $X=G/K$ is equipped with a complete, proper, continuous and $G$-invariant metric $d$. We will freely use the notation of Subsection \ref{s:general assumptions}.
\begin{definition}
	We say that a set $P\subset X$ is locally finite, if $\#(P\cap C)<\infty$ for all bounded Borel sets $C\subset X$.
	Let \[\mathcal{N}^*(X) = \left\{\sum_{x\in P}\delta_x\setdiv P\subset X\text{ locally finite}\right\}.\]
	For $\mu=\sum_{x\in P}\delta_x\in \mathcal{N}^*(X)$ we denote by $\supp(\mu)$ the support of $\mu$ (i.e. $P$).
	We equip the space $\mathcal{N}^*(X)$ with the weak-$*$ topology with respect to $C_c(X)$.
\end{definition}
Note that $G$ acts measurably on $\mathcal{N}^*(X)$ by pushforwards.
As the metric space $X$ is complete, proper and separable, the weak-$*$ topology and the weak-$\#$ topology,  defined in  \cite[Definition 9.1.II]{DaleyVere-Jones2}, agree by \cite[A.2.6.1]{DaleyVere-Jones1}. Thus we have:
\begin{proposition}[Prohorov, {\cite[A.2.6.III]{DaleyVere-Jones1}}]\label{prop:N_star_std_Borel}
	The space $\mathcal{N}^*(X)$ is a standard Borel space, and if $\mathcal{B}_b(X)$ denotes the set of all bounded Borel subsets in $X$, then $\mathcal{B}(\mathcal{N}^*(X))$ is generated by the family of maps $(\pi_A)_{A\in \mathcal{B}_b(X)}$, given by
	\[\pi_A: \mathcal{N}^*(X)\to [0,\infty),\quad \pi_A(\mu)\coloneqq \mu(A).\]
\end{proposition}

\begin{definition}
	Let $(\Omega, \mathbb{P})$ be a probability space equipped with a measurable $G$-action such that $\mathbb{P}$ is $G$-invariant.
	A \emph{stationary point process} in $X$ is a $G$-equivariant map
	\[\Lambda: \Omega\to \mathcal{N}^*(X),\quad \omega\mapsto \Lambda_\omega.\]
\end{definition}
This definition is equivalent to the usual definition of stationary point process. In the sequel all point processes will be stationary, even if not explicitly stated.

By Proposiiton \ref{prop:N_star_std_Borel} any measurable set $B\subset X$ defines a random  variable
\[\Lambda(B):\Omega\to \R,\quad \omega\mapsto \sum_{x\in \supp(\Lambda_\omega)}\delta_x(B) = \#(B\cap \supp(\Lambda_\omega)).\]
If $f:G\to \C$ is measurable, bounded with bounded support, the \emph{linear statistic} of $f$ is the $\C$-valued random variable $\Lambda(f)$ given by 
\[\Lambda(f): \Omega\to \R,\quad \omega\mapsto \Lambda_\omega(f) = \int_X f(x)d\Lambda_\omega(x) = \sum_{x\in \supp{\Lambda_\omega}}f(x).\]

A point process $\Lambda$ in $X$ is called \emph{locally $L^p$} (for $1\leq p\leq \infty$) if for any bounded measurable set $B\subset X$ the $\R$-valued random variable $\Lambda(B)$ is in $L^p$. 
If $\Lambda$ is locally $L^2$, then we also call $\Lambda$ \emph{locally square integrable}.

We will see in Section \ref{s:deterministic_and_random_density} below that all point processes associated to random sphere packings are locally $L^\infty$, hence locally square integrable.

\subsection{Moments and Autocorrelation}
\begin{lemma}\label{l:palm_formula}
	Let $\Lambda$ be a stationary locally $L^1$ point process in $X$. Then there is a constant $i(\Lambda)\geq0$ such that
	\[\mathbb{E}(\Lambda(B)) = i(\Lambda)m_X(B)\]
	for all Borel sets $B$.
\end{lemma}
\begin{proof}
	The map $B\mapsto \mathbb{E}[\Lambda(B)]$ is a $G$-invariant Borel measure,
	as 
	\begin{align*}
		\mathbb{E}[\Lambda(gB)] &= \int_\Omega \Lambda_\omega(gB)d\mathbb{P}(\omega) = \int_\Omega \Lambda_{g^{-1}\omega}(B)d\mathbb{P}(\omega)\\ 
		&= \int \Lambda_\omega(B)d\mathbb{P}(\omega) = E[\Lambda(B)].
	\end{align*}
	Hence our claim holds.
\end{proof}
\begin{definition}
	The constant $i(\Lambda)$ above is called the \emph{intensity} of $\Lambda$.
\end{definition}

The existence of approximate identities in $C_c(G, K)$ directly implies the following lemma.
\begin{lemma}\label{l:uniqueness_invariant_mesures}
	The set $\{f*g\mid f,g\in C_c(G,K)\}$ is dense in $C_c(G,K)$. 
	In particular, if $\mu,\nu$ are two bi-$K$-invariant Borel measures on $G$ such that $\mu(f*g) = \nu(f*g)$ for all $f,g\in C_c(G,K)$, then $\mu = \nu$.
\end{lemma}
Let \[\mathcal{L}^\infty_{\mathrm{bnd}}(G)\coloneqq \{f:G\to \C\mid f\text{ measurable, bounded with bounded support}\}\] and note that this is a convolution algebra with involution.
\begin{proposition}[Björklund--Byléhn,\cite{MichaelMattiasArXiv}]\label{p:autocorrelation_and_moments}
	Let $\Lambda$ be a locally square integrable point process in $X$. For all measurable $b:X\to [0,\infty)$ with $m_X(b) = 1$ and bounded support and Borel sections $\sigma:X\to G$, we have
	\[\mathbb{E}\left[\int \int f^**g(\sigma(x)^{-1}\sigma(y))b(x)d\Lambda(y)d\Lambda(x)\right] = \mathbb{E}[\overline{\Lambda(f\circ \sigma)}\Lambda(g\circ \sigma)]\]
	and 
	\[\mathbb{E}\left[\int \int f^**g(\sigma(x)^{-1}\sigma(y))b(x)d\Lambda(y)d\Lambda(x)\right] -i(\Lambda)^2m_G(f^**g) = \mathrm{Cov}(\Lambda(f\circ \sigma),\Lambda(g\circ \sigma))\]
	for all right-$K$-invariant $f,g\in \mathcal{L}_{\mathrm{bnd}}^\infty(G)$.
\end{proposition}
\begin{proof}
	Let $f,g\in \mathcal{L}_{\mathrm{bnd}}^\infty(G)$. We note that
	\begin{align*}
		f^**g(\sigma(x)^{-1}\sigma(y)) &= \int \overline{f(h^{-1})}g(h^{-1}\sigma(x)^{-1}\sigma(y))dm_G(h)\\
		&= \int \overline{f(h)}g(h\sigma(x)^{-1}\sigma(y))dm_G(h)\\
		&= \int \overline{f(h\sigma(x))}g(h\sigma(y))dm_G(h)\\
	\end{align*}
	and thus
	\begin{align*}
		\mathbb{E}&\left[\int \int f^**g(\sigma(x)^{-1}\sigma(y))b(x)d\Lambda(y)d\Lambda(x)\right]\\
		&=\mathbb{E}\left[  \int_G \int_X\int_X \overline{f(h\sigma(x))}g(h\sigma(y))b(x)d\Lambda(y)d\Lambda(x)\right]\\
		&=  \int_G\mathbb{E}\left[\int_X \int_X  \overline{f(h\sigma(x))}g(h\sigma(y))b(x)d\Lambda(y)d\Lambda(x)\right]dm_G(h)\\
		&= \int_G\mathbb{E}\left[\int_X \int_X  \overline{f(\sigma(hx))}g(\sigma(hy))b(x)d\Lambda(y)d\Lambda(x)\right]dm_G(h)\\
		&= \int_G\mathbb{E}\left[\int_X \int_X  \overline{f(\sigma(x))}g(\sigma(y))b(h^{-1}x)dh_*\Lambda(y)dh_*\Lambda(x)\right]dm_G(h)\\
		&= \int_G\int_\Omega\int_X \int_X  \overline{f(\sigma(x))}g(\sigma(y))b(h^{-1}x)d\Lambda_{h\omega}(y)d\Lambda_{h\omega}(x)d\mathbb{P}(\omega)dm_G(h)\\
		&= \int_G\int_\Omega\int_X \int_X  \overline{f(\sigma(x))}g(\sigma(y))b(h^{-1}x)d\Lambda_{\omega}(y)d\Lambda_{\omega}(x)d\mathbb{P}(\omega)dm_G(h)\\
		&= m_G(b) \mathbb{E}[\Lambda(f)\Lambda(g)]\\
		&= \mathbb{E}[\Lambda(f)\Lambda(g)].
	\end{align*}
	Note that 
	\begin{align*}
		\mathrm{Cov}(\Lambda(f\circ\sigma),\Lambda(g\circ\sigma)) &= \mathbb{E}[\overline{\Lambda(f\circ\sigma)}\Lambda(g\circ\sigma)]-\mathbb{E}[\overline{\Lambda(f\circ\sigma)}]\mathbb{E}[\Lambda(g\circ\sigma)]\\
		&= \mathbb{E}[\overline{\Lambda(f\circ\sigma)}\Lambda(g\circ\sigma)]-i(\Lambda)^2m_G(\overline{f\hspace{2pt}})m_G(g).
	\end{align*}
	As $m_G(f^**g) = m_G(\overline{f\hspace{2pt}})m_G(g)$, the claim follows.
\end{proof}
\begin{corollary}[Björklund--Byléhn,\cite{MichaelMattiasArXiv}]
	Let $\Lambda$ be a locally square integrable point process in $X$.
	For any choice of $b$ and $\sigma$ as in Proposition \ref{p:autocorrelation_and_moments} we can define bi-$K$-invariant Radon measures $\eta_\Lambda$ and $\eta_\Lambda^+$ on $G$ by
	\[
	\eta_\Lambda^+(f)\coloneqq \mathbb{E}\left[\int \int f^\natural(\sigma(x)^{-1}\sigma(y))b(x)d\Lambda(y)d\Lambda(x)\right]
	\]
	and
	\[
	\eta_\Lambda(f)\coloneqq \mathbb{E}\left[\int \int f^\natural(\sigma(x)^{-1}\sigma(y))b(x)d\Lambda(y)d\Lambda(x)\right] -i(\Lambda)^2m_G(f)
	\]
	for all $f\in C_c(G)$. These measures are independent of $b$ and $\sigma$.
\end{corollary}
\begin{proof}
	This follows directly from Lemma \ref{l:uniqueness_invariant_mesures} and Proposition \ref{p:autocorrelation_and_moments}.
\end{proof}
Now the following corollary follows directly from Lemma \ref{l:lf_measure_are_distributions}.
\begin{corollary}\label{c:autocorrelation_distributions_exist}
	If $G$ is a Lie group, $\eta_\Lambda^+$ and $\eta_\Lambda$ induce distributions on $G$.
\end{corollary}
\begin{definition}[Björklund--Hartnick--Pogorzelski, \cite{BHP2}, Björklund--Hartnick, \cite{BjorklundHartnickHyperuniformity}]
	Let $\Lambda$ be a locally square integrable point process in $X$. The bi-$K$-invariant Borel measures $\eta_\Lambda^+$ and $\eta_\Lambda$ on $G$ are called \emph{autocorrelation measure} and and \emph{reduced autocorrelation measure} of $\Lambda$.
	The distributions $T_\Lambda^+$ and $T_\Lambda$ induced by $\eta_\Lambda^+$ and $\eta_\Lambda$ are called \emph{autocorrelation distribution} and \emph{reduced autocorrelation distribution} of $\Lambda$.
\end{definition}
\begin{proposition}
	Let $\Lambda$ be a locally square integrable point process in $X$.
	The measures $\eta_\Lambda$ and $\eta_\Lambda^+$ are positive-definite and thus the distributions $T_\Lambda^+$ and $T_\Lambda$ are positive definite.
\end{proposition}
\begin{proof}
	Let $g\in C_c(G)$ be right-$K$-invariant. 
	Then the Proposition \ref{p:autocorrelation_and_moments} implies that 
	\[
	\eta_\Lambda^+(g^**g)\geq 0.
	\]
	If $f\in C_c(G)$, then 
	\begin{align*}
		\eta_\Lambda^+(f^**f) &= \int_G\int_K\int_K f^**f(kgk)\,dm_K(k_1)\,dm_K(k_2)\,d\eta_\Lambda^+(g)\\
		&= \int_G\int_K\int_K \int_G \overline{f(h^{-1})}f(h^{-1}k_1gk_2)\,dm_G(h)\,dm_K(k_1)\,dm_K(k_2)\,d\eta_\Lambda^+(g)\\
		&= \int_G\int_K\int_K \int_G \overline{f(h^{-1}k_1^{-1})}f(h^{-1}gk_2)\,dm_G(h)\,dm_K(k_1)\,dm_K(k_2)\,d\eta_\Lambda^+(g)\\
		&= \int_G \int_K\int_K\int_G \overline{f(h^{-1}k_1)}f(h^{-1}gk_2)\,dm_G(h)\,dm_K(k_1)\,dm_K(k_2)\,d\eta_\Lambda^+(g)\\
		&= \int_G \int_G \overline{f'(h^{-1})}f'(h^{-1}g)\,dm_G(h)\,d\eta_\Lambda^+(g)\\
		&= \int_G (f')^**(f')(g)\,d\eta_\Lambda^+(g)\\
		&= \eta_\Lambda^+((f')^**(f'))\geq 0,
	\end{align*}
	where we note that the function $f'$ on $G$ defined by
	\[
	f'(g)= \int_Kf(gk)dm_K(k)
	\]
	is right-$K$-invariant and in $C_c(G)$.
	The same calculation applied to $\eta_\Lambda$ shows that $\eta_\Lambda$ is positive definite.
\end{proof}
\section{Densities of random and deterministic packings}\label{s:deterministic_and_random_density}
In this section we introduce random sphere packings and define their density. We explain how they are connected to point processes and how the notion of density of random sphere packings and of a deterministic sphere packings are related.
We additionally give several examples of random sphere packings. As in the previous section we assume that $(G,K)$ is a Lie group Gelfand pair and that $X=G/K$ is equipped with a complete, proper, continuous and $G$-invariant metric $d$.
\subsection{Densities of random packings}
\begin{definition}
	Let $\mathfrak{F}(X)$ denote the set of closed subsets of $X$ and let $(\Omega, \mathcal{A} ,\mathbb{P})$ be a probability space on which $G$ acts measurably such that $\mathbb{P}$ is $G$-invariant.
	A $G$-equivariant map $\Lambda:\Omega\to \mathfrak{F}(X)$ is a \emph{stationary random closed set} if 
	\[\{\omega\in \Omega\mid \Lambda(\omega)\cap C \neq \emptyset\}\in \mathcal{A}\]
	for any compact $C\subset X$.
	
	Denote by $\mathfrak{G}(X)$ the set of open subsets of $X$. A $G$-equivariant map $\Lambda:\Omega\to \mathfrak{G}(X)$ is called a \emph{stationary random open set} if 
	\[\{\omega\in \Omega\mid C\subset \Lambda(\omega)\}\in \mathcal{A}\] 
	is measurable for all $C\in \mathfrak{F}(X)$.
\end{definition}

\begin{remark}
	The set $\mathfrak{F}(X)$ can be equipped with a a compact second-countable Hausdorff topology, the Chabauty-Fell topology, see for instance appendix C in \cite{Molchanov}. This topology is generated by the subbasis
	\[U^C\coloneqq \{A\in \mathfrak{F}(X)\mid A\cap C=\emptyset\}\qquad U_V\coloneqq \{A\in \mathfrak{F}(X)\mid A\cap V\neq \emptyset\},\]
	where $C$ ranges over the compact subsets of $X$ and $V$ over the open subsets of $X$. 
	The condition in the definition above encodes measurability of $\Lambda$ with respect to the Borel $\sigma$-algebra of the Fell topology.
\end{remark}

See \cite{Molchanov} for more information on random closed/open sets, for invariance properties under group actions see specifically Proposition 1.3.30 and Section 1.5.1. Note in particular that the natural action of $G$ on $\mathfrak{F}(X)$ is continuous.

\begin{definition}
	Let $r>0$.
	We call $P\subset X$ is \emph{$r$-uniformly discrete} if $d(x,y)\geq r$ for all $x\neq y\in P$ and denote the set of $r$-uniformly discrete subsets of $X$ by $\UD_r(X)$.
	We set \[\mathcal{N}^*_r(X)=\{\mu\in \mathcal{N}^*(X)\mid \supp(\mu)\in\UD_r(X)\}\]
	and call a point process $\Lambda$ \emph{$r$-uniformly discrete}, if 
	$\Lambda_\omega\in \mathcal{N}_r^*(X)$ for all $\omega\in \Omega$ or equivalently $\Lambda_\omega(B_r(x))\in \{0,1\}$ for all $x\in X$ and $\omega\in \Omega$.
	We equip $\UD_r(X)\subset \mathfrak{F}(X)$ with the subspace topology.
\end{definition}

We think of $2r$-uniformly discrete point processes $\Lambda$ as random sphere packings, by interpreting the points in $\Lambda$ as centers of spheres of radius $r$.

\begin{lemma}\label{l:uniform_discrete_implies_locally_Lp}
	Let $\Lambda$ be a $r$-uniformly discrete point process for some $r>0$. Then $\Lambda$ is locally $L^\infty$ (and thus  locally square integrable and locally $L^1$).
\end{lemma}
\begin{proof}
	Let $B\subset X$ be bounded and measurable and choose $R>0$ such that $B\subset B(x_0, R)$.
	We have for any $r$-uniformly discrete $P\subset X$ that 
	\[\#(P\cap B) = \frac{m_X(\bigcup_{x\in P\cap B}B(x,r/2))}{m_X(B(x_0,r))}\leq \frac{m_X(B(x_0, R+r/2))}{m_X(B(x_0,r)}.\]
	Hence, if $\Lambda$ is a $r$-uniformly discrete point process, we have
	\[\Lambda_\omega(B) = \#(P_\omega\cap B) \leq \frac{m_X(B(x_0, R+r/2))}{m_X(B(x_0,r)}\]
	for any $\omega\in \Omega$. Here $P_\omega=\supp(\Lambda_\omega)$.
	Thus $\Lambda$ is locally $L^\infty$.
\end{proof}
\begin{corollary}
	A $r$-uniformly discrete stationary point process has a (reduced) autocorrelation measure and distribution.
\end{corollary}

\begin{proposition}[{\cite[Proposition 3.2]{BHP2}}]
	The bijection $\mathcal{N}_r^*(X)\to \UD_r(X)$ is a homeomorphism and $\UD_r(X)$ is compact and metrizable.
\end{proposition}

\begin{corollary}
	If $\Lambda$ is a $r$-uniformly discrete point process, then 
	\[\supp(\Lambda):\Omega\to \UD_r(X),\ \omega\mapsto \supp(\Lambda_\omega)\]
	is a random closed set.
\end{corollary}
\begin{proposition}
	If $\Lambda$ is a $2r$-uniformly discrete point process, then the map
	\[\supp(\Lambda)^r:\Omega\to \mathfrak{G}(X),\ \omega\mapsto \bigcup_{x\in \supp(\Lambda_\omega)}B(x,r)\]
	is a stationary random open set.
\end{proposition}
\begin{proof}
	Define a function $\mathrm{dist}:\Omega\times X\to [0,\infty]$ by $\mathrm{dist}(\omega, x)\coloneqq \inf_{y\in \supp(\Lambda_\omega)}d(x,y)$. By \cite[Theorem 1.3.14]{Molchanov}, the map $\supp(\Lambda)$ is Effros measurable, see \cite[Definition 1.3.1]{Molchanov}. Hence by the proof of \cite[Theorem 1.3.3]{Molchanov}, the function $\mathrm{dist}(\omega, x)$ is jointly measurable.
	Note that 
	\[\supp(\Lambda)^r(\omega) = \{x\in X\mid \mathrm{dist}(\omega,x)< r\}.\]
	If $C\in \mathfrak{F}(X)$, then we can choose a countable dense subset $D\subset C$ (as $X$ is second-countable). Now 
	\[\{\omega\in \Omega\mid C\subset \supp(\Lambda)^r(\omega)\} = \{\omega\in \Omega\mid \inf_{x\in D}\mathrm{dist}(\omega,x)<r\}\]
	is measurable.
\end{proof}

The following definition is a direct generalization of the definition of density for random sphere packings in hyperbolic space, introduced by Bowen and Radin in \cite{Bowen2003}.
\begin{definition}
	The \emph{density} of a $2r$-uniformly discrete point process $\Lambda$ is defined as
	\[D(\Lambda) \coloneqq \mathbb{P}(x_0\in\supp(\Lambda)^r).\]
\end{definition}
\begin{proposition}\label{p:formula_density}
	Let $\Lambda$ be a $2r$-uniformly discrete point process and $x\in X$. Then
	\[D(\Lambda) = \mathbb{E}[\Lambda(B(x, r))] = i(\Lambda)m_X(B(x_0,r)).\]
\end{proposition}
\begin{proof}
	Choose $g\in G$ with $g^{-1}x_0=x$ and observe 
	\begin{align*}
		D(\Lambda) &= \mathbb{P}(x_0\in \supp(\Lambda^r)) = \int \chi\{\mathrm{dist}(x_0,\supp(\Lambda))(\omega)\leq r\} d\mathbb{P}(\omega)\\ 
		&= \int \mathbbm{1}\{\supp(\Lambda_\omega)\cap B(x_0,r)\neq \emptyset\}d\mathbb{P}(\omega) =  \int \mathbbm{1}\{\#(\supp(\Lambda_\omega)\cap B(x_0,r))=1\}d\mathbb{P}(\omega)\\
		&= \mathbb{E}[\Lambda(B(x, r))].\qedhere
	\end{align*}
\end{proof}

\begin{definition}\label{d:definition_prob_optimal_density}
	For $r>0$ we define the \emph{probabilistic optimal $r$-packing density of $X$} as
	\[\triangle_{\mathrm{prob}}(r,X)\coloneqq \sup \{D(\Lambda)\mid \Lambda\text{ $2r$-uniformly discrete point process in $X$}\}.\]
\end{definition}

\subsection{Densities of deterministic packings}\label{s:deterministic_packings}
We will now explain how our results about the intensity of $2r$-uniformly discrete point processes relate to the density of (deterministic) sphere packings. We again interpret these point sets as sets of centers of spheres in a sphere packing.

\begin{definition}
	Let $P\subset X$ be $2r$-uniformly discrete.
	\begin{enumerate}[(i)]
		\item The \emph{lower $r$-density of $P$ centered at $x\in X$} is defined as
		\[\underline{D_r}(P,x)\coloneqq \liminf_{R\to \infty}\frac{m_X(B(x,R)\cap \bigcup_{y\in P} B(y,r))}{m_X(B(x,R))}.\]
		\item The \emph{higher $r$-density of $P$ centered at $x\in X$} is defined as
		\[\overline{D_r}(P)\coloneqq \limsup_{R\to \infty}\frac{m_X(B(x,R)\cap \bigcup_{y\in P}B(y,r))}{m_X(B(x,R))}.\]
		\item If $\underline{D_r}(P,x) = \overline{D_r}(P,x)$ we say that \emph{$P$ has $r$-density centered at $x\in X$} and write 
		\[D_r(P,x)\coloneqq \underline{D_r}(P,x) = \overline{D_r}(P,x).\]
		\item If $P$ has $r$-density centered at $x$ for every $x\in X$ and $D_r(P,x)=D_r(P,y)$ for all $x,y\in X$, we say that \emph{$P$ has $r$-density}
		\[D_r(P) \coloneqq D_r(P,x_0).\]
	\end{enumerate}
\end{definition}
The following approach to sphere packing was pioneered by Bowen and Radin \cite{Bowen2003},\cite{Bowen2004}, in the context of hyperbolic spaces; see \cite{HartnickWackenhuthDensity} for terminology and proofs in the general setting.
\begin{definition}
	Let $(G_t)_{t>0}$ be a sequence of subsets of $G$ of positive measures, let $\mathcal{F}$ be a set of measurable functions on $\UD_{2r}(X)$ 
	such that the map \[
		\mathrm{Prob}(\UD_{2r}(X)) \to \C^\mathcal{F},\quad \nu \mapsto (\nu(f))_{f\in \mathcal F}
	\] is injective,
	and let $\mu$ be an ergodic probability measure on $\UD_{2r}(X)$.
	\begin{enumerate}[(i)]
		\item We say that $P_0\in \UD_{2r}(X)$ is $(\mu, \mathcal{F}, (G_t)_{t>0})$-generic, if 
			\[\lim_{t\to\infty} \frac{1}{m_G(G_t)}\int_{G_t}f(gP_0)dm_G(G) = \mu(f)\quad (f\in \mathcal{F}).\]
		\item We say that $P_0\in \UD_{2r}(X)$ is invariantly $(\mu, \mathcal{F}, (G_t)_{t>0})$-generic, if 
			\[\lim_{t\to\infty} \frac{1}{m_G(G_t)}\int_{G_t}f(ghP_0)dm_G(G) = \mu(f)\quad (f\in \mathcal{F}, h\in G).\]
	\end{enumerate}
\end{definition}

\begin{remark}
For the following see \cite{HartnickWackenhuthDensity}.
	\begin{enumerate}[(i)]
		\item If $P_0$ is $(\mu, \mathcal{F}, (G_t)_{t>0})$-generic, then $\mu$ is uniquely determined by $P_0$ (and $\mathcal{F}$ and $(G_t)_{t>0}$) and supported on the orbit closure of $P_0$.
		\item Conversely, every ergodic probability measure $\mu$ on $\UD_{2r}(X)$ is supported on an orbit closure, and if the pointwise ergodic theorem holds for $((G_t)_{t>0}, \mathcal{F})$, then $\mu$-almost every point in this orbit closure is $(\mu, \mathcal{F}, (G_t)_{t>0})$-generic.
		\item If $G$ is amenable and $(G_t)_{t>0}$ is a F\o lner sequence, the existence of an invariantly generic $P_0\in \UD_{2r}(X)$ comes for free. 
			If $G$ is non-amenable, this is not automatically the case. This leads to the following definition.
	\end{enumerate}
\end{remark}
\begin{definition}\label{d:invariant_pointwise_ergodic_theorem}
	We say that the \emph{invariant pointwise ergodic theorem} holds for the pair $((G_t)_{t>0}, \mathcal{F})$ if for every $G$-invariant $\mu\in \mathrm{Prob}(\UD_{2r}(X))$ on $\UD_{2r}(X)$ there is a conull set $\Omega_0\subset \UD_{2r}(X)$ of invariantly $(\mu, \mathcal{F}, (G_t)_{t>0})$-generic point sets.
\end{definition}

From now on we fix $G_t\coloneqq \pi^{-1}(B(x_0, t))$ to be preimages of balls in $X$ and choose $\mathcal{F}$ to be the set of all Riemann integrable compactly supported functions on $\UD_{2r}(X)$.

\begin{example}\label{e:Gelfand_pairs_with_invariant_pointwise_ergodic_theorem}
	The following Gelfand pairs have invariant pointwise ergodic theorems.
	\begin{enumerate}[(i)]
		\item $(\R^n, \{0\})$ and $(\R^n\rtimes SO(n), SO(n))$, both with the Euclidean metric, by the ergodic theorem for amenable groups.
		\item $(\mathcal{H}_n\rtimes U(n), U(n))$ with $d$ the Cygan-Koranyi metric on $\mathcal{H}_n$, by the ergodic theorem for amenable groups.
		\item $(G, K)$ such that $G$ is a (connected) simple Lie group with finite center and no compact factors and $K\subset G$ is a maximal compact subgroup with $d$ the Cartan-Killing distance on $G/K$. This follows from \cite[Theorem 1.2.]{GorodnikNevo} by Gorodnik and Nevo.
	\end{enumerate}
\end{example}

\begin{definition}\label{d:generically_measured}
	We say that $P_0\in \UD_{2r}(X)$ is \emph{generically measured} if there exists an ergodic $\mu_{P_0}\in \mathrm{Prob}(\UD_{2r}(X))$ (necessarily unique) such that $P_0$ is invariantly $(\mu_{P_0}, \mathcal{F}, (G_t)_{t>0})$-generic. 
\end{definition}
\begin{example} See Subsection \ref{s:Examples} for more in-depth discussion of the following examples.
\begin{enumerate}[(i)]
	\item If $P_0\in \UD_{2r}(X)$ is the orbit of a lattice in $G$ and $(G,K)$ is one of the Gelfand pairs in Example \ref{e:Gelfand_pairs_with_invariant_pointwise_ergodic_theorem}, then $P_0$ is generically measured.
	\item Similarly, if $P_0\in \UD_{2r}(X)$ consists of multiple orbits of a lattice $\Gamma < G$, then $P_0$ is generically measured.
	\item If $P_0\in \UD_{2r}(X)$ is the orbit of a regular cut-and-project set (see Example \ref{e:cut-and-project sets} below), then $P_0$ is generically measured.
\end{enumerate}
\end{example}
For a proof of the following proposition see Bowen and Radin, \cite{Bowen2004}; in the specific form given here we give a proof in \cite{HartnickWackenhuthDensity}.
\begin{proposition}
	If $P_0\in \UD_r(X)$ is generically measured, then there exists a stationary point process $\Lambda$ with distribution $\mu_{P_0}$, and for any such process $\Lambda$ and every $x\in X$ we have 
	\[D(\Lambda^r) = D_r(P_0, x).\]
	In particular, $P_0$ has a well-defined $r$-density which is independent of the base point.
\end{proposition}

\begin{definition}\label{d:Bowen_Radin_optimal_density}
The \emph{Bowen-Radin optimal density} $\triangle(r,X)$ is defined as 
\[\triangle(r,X):=\sup \{D(P_0, x) \mid x\in X, P_0\in \UD_r(X) \text{ generically measured}\}.\]
\end{definition}

\begin{corollary}\label{c:equiv_defs_optimal_density}
	If the invariant pointwise ergodic theorem holds for $((G_t)_{t>0}, \mathcal{F})$, then 
	\[\triangle(r,X) = \triangle_{\mathrm{prob}}(r,X).\]
\end{corollary}

\begin{remark}
	$\triangle(r,\R^n)$ agrees with the usual definition of the optimal packing density of $\R^n$, as every periodic sphere packing is generically measured.
\end{remark}


\begin{remark}
	If $G$ is compact, there is a very tight connection between $2r$-uniformly discrete point processes $\Lambda:(\Omega, \mathbb{P})\to \mathcal{N}^*_{2r}(X)$ such that $\mathbb{P}$ is ergodic and $2r$-uniformly discrete point sets.
	
	If $P$ is a uniformly discrete point set in $X$, then $G.P = \overline{G.P}\subset \UD_{2r}(X)$, by the compactness of $G$ and the continuity of the $G$-action. Hence there is a unique $G$-invariant measure $\mu_P$ on $G.P$ induced by $m_{G/\mathrm{Stab}_G(P)}$ and $\Lambda_P: G.P \to \mathcal{N}_r^*(X),\ P'\mapsto \sum_{x\in P'}\delta_x$ is a point process such that $\mathbb{P}=\mu_P$ is ergodic.
	
	If  $\Lambda:(\Omega, \mathbb{P})\to \mathcal{N}^*_{2r}(X)$ is a  $2r$-uniformly discrete point processes such that $\mathbb{P}$ is ergodic, then $\Lambda_*\mathbb{P}$ is an ergodic measure on $\mathcal{N}_r^*(X)$. As $\mathcal{N}_r^*(X)$ is a standard Borel space and the orbits of $G$ are closed, \cite[Proposition 2.1.12]{Zimmer} and \cite[Proposition 2.1.10]{Zimmer} imply that $\Lambda_*\mathbb{P}$ is supported on a $G$-orbit, i.e. that there is some $P\in \UD_{2r}(X)$ such that $\Lambda_P$ has the same distribution as $\Lambda$.
\end{remark}
\subsection{Examples}\label{s:Examples}
\begin{example}[Lattices]
One example for $r$-uniformly discrete point processes in homogeneous spaces is given by lattices.
Let $\Gamma\subset G$ be a lattice, i.e. a discrete subgroup such that there is a unique $G$-invariant measure $m_{G/\Gamma}$ on $G/\Gamma$ such that 
\[\int f(g)dm_G(g) = \int_{G/\Gamma} \sum_{\gamma\in \Gamma}f(g\gamma)dm_\Gamma(\gamma)dm_{G/\Gamma}(g\Gamma)\]
for all $f\in C_c(G)$. We set $\mathrm{covol}(\Gamma)\coloneqq m_{G/\Gamma}(G/\Gamma)$ and 
obtain a probability measure $\mathbb{P}=\frac{1}{\mathrm{covol}(\Gamma)}m_{G/\Gamma}$ on $\Omega\coloneqq G/\Gamma$. Consider the map
\[\Lambda^\Gamma:\Omega\to \mathcal{N}^*(X),\quad g\Gamma\mapsto \sum_{x\in g\Gamma.x_0}\delta_x.\]
This is a point process, see for instance \cite[Example 2.2]{MichaelMattiasArXiv}.
Choose an open set $U\subset G$ with $U\cap \Gamma = \Gamma\cap K$. Then $\pi(U)\cap \Gamma.x_0 = \pi(U\cap \Gamma) = x_0$. As $\pi(U)$ is open, we can choose $r>0$ with $B(x_0,r)\subset \pi(U)$ and see that $\Gamma.x_0$ is $r$-uniformly discrete, as
\[
\#(B(\gamma.x_0,r)\cap \Gamma) = \#(B(x_0,r)\cap \gamma^{-1}\Gamma) = \#(B(x_0,r)\cap \Gamma) = 1
\]
for all $\gamma\in \Gamma$.
This implies that $\Lambda^\Gamma$ is $r$-uniformly discrete.
By \cite[Example 2.5]{MichaelMattiasArXiv} the intensity of $\Lambda^\Gamma$ is given by 
\[i(\Lambda^\Gamma) = \frac{1}{\#(\Gamma\cap K)\mathrm{covol}(\Gamma)}\]
and the reduced autocorrelation measure by
\[\eta_\Lambda(f)=i(\Lambda^\Gamma)\left(\frac{1}{\#(\Gamma\cap K)}\sum_{g\in\Gamma}f(g)- i(\Lambda^\Gamma)m_G(f)\right).\]
\end{example}

\begin{example}[Cut-and-project sets]\label{e:cut-and-project sets}
A second class of $r$-uniformly discrete point processes, related to the theory of quasicrystals, is given by orbits of regular cut-and project sets. This class of examples was introduced in the generality we need by Björklund, Hartnick and Pogorzelski in \cite{BHP1, BHP2, BHP3}. See also \cite[Section 7]{BHSiegelTransform} for the perspective we present in the following.
Let $H$ be a lcsc group, $\Gamma\leq G\times H$ a lattice such that the projection $\mathrm{pr}_G:G\times H\to G$ restricted to $\Gamma$ is injective and the projection $\mathrm{pr}_H$ to $H$ satisfies that $\mathrm{pr}_H(\Gamma)$ is dense in $H$. Choose a pre-compact Jordan-measurable subset $W\subset H$ with dense interior and define the \emph{cut-and-project set} \[P_0(g, h) = \mathrm{pr}_G((g,h)\Gamma\cap G\times W).\]
Then for almost all choices $(g, h)\in G\times H$ there is a $G$-invariant Borel probability measure $\mu_{P_0}$ on the \emph{hull} $\Omega_{P_0(g,h)}=\overline{G.\pi(P_0(g,h))}\setminus \{\emptyset\}\subset\UD_{2r}(X)$ (here $r$ depends on the choice of $\Gamma$, $W$) and a factor map
\[((G\times H)/\Gamma,\frac{1}{\mathrm{covol}(\Gamma)}m_{G\times H/\Gamma})\to (\UD_{2r}(X),\mu_{P_0}),\  x\mapsto Y_x,\]
where $Y_x= \{gK\in G/K\mid (g^{-1},e)x\in (K\times W)/\Gamma\}$.
If $W$ is regular, i.e. $\partial W \cap \Gamma = \emptyset$ and $\mathrm{Stab}_H(W)=\{e\}$, then $P_0\coloneqq P_0(e,e)$ is called a \emph{regular model set}, and by \cite{BHP1} the hull $\Omega_{P_0(e,e)}$ is uniquely ergodic and each of its elements is generically measured.

In this case by \cite{BHP1} the autocorrelation of $P_0$ (or rather $\pi(P_0)$) is given by

\[\eta_{\mu_{P_0}}(f)\coloneqq \lim_{t\to \infty}\frac{1}{m_G(G_t)}\sum_{g\in P_0\cap G_t}\sum_{y\in P_0} f(x^{-1}y),\]
if $\pi|_{P_0}$ is injective (if not then one needs to work with weighted point sets, see \cite{BHP2}).
A formula for the spherical diffraction of these sets can be found in \cite{BHP3}.
\end{example}

The study of model sets led us to the proof of our main theorem. We first proved the bounds conjectured by Cohn and Zhao for sphere packings coming from model sets by use of exotic Poisson summation. We then noticed that only the existence of an atom encoding the density at the trivial character was relevant, and that the rest of the summation structure on the spectral side was irrelevant.

\begin{example}[Thinned Poisson point processes]
In the Euclidean case $X=\R^n$, Campos, Jenssen, Michelen and Sahasrabudhe showed in \cite{EuclideanLowerArXiv} that thinning Poisson processes in a controlled way yields $r$-uniformly discrete point processes in $\R^n$ of high intensity. They then used these processes to prove a new asymptotically best lower bound on densities of sphere packings in $\R^n$.

In \cite{HyperbolicLowerArXiv} this approach was generalized to the case $X=\mathbb{H}^n$ of hyperbolic $n$-space by Fernández, Kim, Liu and Pikhurko.
\end{example}
\section{General linear programming bound}\label{s:bounds}
In this section we will prove the linear programming bound for the optimal packing density of convenient commutative spaces.

\subsection{Spherical diffraction}
For this subsection we only impose the assumptions from \ref{s:general assumptions}. Recall that we assume that $(G, K)$ is a Lie group Gelfand pair and that $X=G/K$ is equipped with a complete, proper, continuous and $G$-invariant metric $d$.
Let $\Lambda$ be a locally square integrable point process in $X$.
\begin{definition}
The positive Borel measure $\widehat{\eta}_\Lambda^+$ on $PS(G,K)$ is called the \emph{spherical diffraction} of $\Lambda$ and the positive Borel measure $\widehat{\eta}_\Lambda$ on $PS(G,K)$ is called the \emph{reduced spherical diffraction}.
\end{definition}
Note that the measures $\widehat{\eta}_\Lambda^+$ and $\widehat{\eta}_\Lambda$  exist by Theorem \ref{t:Godement-Plancherel_for_measures}.
\begin{remark}
	It is possible to define more general notions of diffraction for all type 1 unimodular Lie groups (with a fixed compact subgroup $K$) by applying Bonnets Plancherel theorem \cite{BonnetPlancherelTheorem} to the positive definite distribution induced by $\eta_\Lambda^+$. Other available Plancherel theorems include Penneys Plancherel theorem \cite{DistributionalPlancherelPenney} for distributions.
	When the center of the universal enveloping algebra of the Lie algebra of $G$ contains an elliptic element, then a more concrete Plancherel theorem by Aarnes is available, see \cite{EllipticCentralPlancherel}.
	
	In \cite{Thomas} Thomas proves a Bochner theorem for positive-definite spherical distributions on generalized Gelfand pairs $(G,H)$. In fact, Thomas shows that the existence of such a Bochner theorem is equivalent to $(G,H)$ being a generalized Gelfand pair, see also {\cite[Proposition 8.3.2]{vanDijk}} for a textbook account.
	
	We will not pursue these matters for now, as they will not lead to workable linear programming bounds for packings.
\end{remark}
\begin{lemma}\label{l:formula_diffraction_and_reduced_diffraction}
	We have
	\[\widehat{\eta}_\Lambda^+ = \widehat{\eta}_\Lambda + i(\Lambda)^2\delta_\mathbf{1}.\]
\end{lemma}
\begin{proof}
	We have $\eta_\Lambda^+ = \eta_\Lambda+i(\Lambda)^2m_G$. The uniqueness of the Plancherel-Godement transform in Theorem \ref{t:Godement-Plancherel_for_measures} implies that $\widehat{m_G} = \delta_\mathbf{1}$, as
	\[\delta_\mathbf{1}(\widehat{f}) = \widehat{f}(\mathbf{1}) = \int f(g)\mathbf{1}(g^{-1})dm_G(g) = m_G(f)\]
	for all $f\in C_c(G,K)^2$.
	The uniqueness in Theorem \ref{t:Godement-Plancherel_for_measures} then implies the claim.
\end{proof}

\subsection{The linear programming bound}
Assume that $(G,K, d, \mathcal{S}(G, K))$ is a convenient Gelfand pair.
\begin{example}
	Recall that the following algebras are Schwartz-like:
	\begin{enumerate}[(i)]
		\item $C_c(G,K)^2=\lspan\{f*g\mid f,g\in C_c(G,K)\}$ for a general Lie group Gelfand pair.
		\item $\mathscr{S}(\R^n)$.
		\item $\mathscr{S}^1(G,K)$, where $G$ is a semisimple Lie group with finite center and no compact factors, $K$ a maximal compact subgroup.
		\item $C^\infty(G,K)$ for Riemannian symmetric pairs $(G,K)$ of compact type.
		\item $\mathscr{S}(\mathcal{H}_n\rtimes U(n), U(n))$.
	\end{enumerate}
\end{example}
\begin{theorem}\label{t:lp_bound}
	Let $(G, K, d, \mathcal{S}(G, K))$ be a convenient commutative space and let $\Lambda$ be a $2r$-uniformly discrete point process in $X = G/K$. Assume that $f$ is a witness function.
	Then
	\begin{align*}
		\label{e:general_estimate}i(\Lambda)\leq \frac{f(e)}{\widehat{f}(\mathbf{1})}.
	\end{align*}
	Hence 
	\[\triangle(r,X)\leq m_X(B(x_0,r))\frac{f(e)}{\widehat{f}(\mathbf{1})}.\]
\end{theorem}
\begin{proof}
	Lemma \ref{l:formula_diffraction_and_reduced_diffraction} implies that 
	\[
	\widehat{\eta}_\Lambda^+(\widehat{f}) = \widehat{\eta}_\Lambda(\widehat{f}) + i(\Lambda)^2\widehat{m_G}(\widehat{f}) \geq i(\Lambda)^2\widehat{f}(\mathbf{1}),
	\]
	where we have used condition (W2) and the fact that $\widehat{\eta}_\Lambda$ is a positive measure.
	Assume first that the support of $f$ is compact.
	Fix $R>0$ such that $m_X(B(x_0,R)) = 1$ and set $b\coloneqq\chi_{B(x_0,R)}$. We set $P\coloneqq \supp(\Lambda)$.
	Then
	\begin{align*}
		\eta_\Lambda^+(f) &= \mathbb{E}\left[\sum_{y\in P}\sum_{x\in P\cap B(x_0,R)}f(\sigma(x)^{-1}\sigma(y))\right]
		\leq \mathbb{E}\left[\sum_{x\in P\cap B(x_0,R)}f(\sigma(x)^{-1}\sigma(x))\right]\\
		&= \mathbb{E}\left[f_n(e)\sum_{x\in P\cap B(x_0,R)}1\right] = f(e)\mathbb{E}\left[\#(P\cap B(x_0,R))\right] = f_n(e)\mathbb{E}\left[\Lambda(B(x_0,R))\right]\\ 
		&= f(e)i(\Lambda) m_{X}(B(x_0,R)) = f(e)i(\Lambda),
	\end{align*}
	using property (W1) for the inequality and Lemma \ref{l:palm_formula} in the second to last equality.
	Hence $\widehat{\eta}_\Lambda^+(\widehat{f})\leq f(e)i(\Lambda)$.
	
	If the support of $f$ is non-compact, set $f_n\coloneqq g_nf$, where $(g_n)_{n\geq 1}$ is a sequence in $C_c^\infty(G, K)$ as in  Definition \ref{d:Schwartz-like}.
	Note that the functions $f_n$ satisfy (W1) and that $f(e) = f_n(e)$ for all $n$.
	By the calculation above 
	\[\widehat{\eta}_\Lambda^+(\widehat{f}) = T_{\eta_\Lambda^+}(f) = \lim_{n\to \infty} \eta_\Lambda^+(f_n)\leq \lim_{n\to \infty} i(\Lambda)f_n(e) = i(\Lambda)f(e).\]
	Thus in either case
	 \[i(\Lambda)f(e)\geq\ \widehat{\eta}_\Lambda^+(\widehat{f})\geq i(\Lambda)^2\widehat{f}(\mathbf{1})\] 
	and we obtain
	\[
	i(\Lambda) \leq \frac{f(e)}{\widehat{f}(\mathbf{1})}
	\]
	as $i(\Lambda)>0$.
	Now the bound on $\triangle_{\mathrm{prob}}(r,X)$ follows directly from Proposition \ref{p:formula_density}. And thus by Corollary \ref{c:equiv_defs_optimal_density} we obtain the result.
\end{proof}

\begin{remark}
	In the absence of a invariant pointwise ergodic theorem, the method above still yields estimates for the probabilistic optimal packing density.
\end{remark}
\subsection{A curious observation}
\begin{proposition}
	Let $G$ be a (connected) simple Lie group without compact factors and finite center and $K$ a maximal compact subgroup. Let $d$ be the Cartan-Killing metric and let $G=KAN$ be an Iwasawa decompositions of $G$ as in Subsection \ref{ss:noncompt_ss_Lie_harmonic}. Then $\mathfrak{a}$ with metric $d_{\mathfrak{a}}$ induced by the Killing form and the Haar measure $m_\mathfrak{a}$ with the standard normalization is a Euclidean space and $(\mathfrak{a}, \{0\}, d_\mathfrak{a}, \mathscr{S}(\mathfrak{a}))$ is a convenient Gelfand pair. If $f\in \mathcal{W}_r(G/K)$, then the Abel transform $\mathcal{A}f$ is in $\mathcal{W}_r(\mathfrak{a})$.
\end{proposition}
\begin{proof}
	\cite[Chapter VI, Exercise B2.(iv)]{Helgason2} states that
	$d(x_0,nax_0)\geq d(x_0, ax_0)$
	for all $a\in A, n\in N$.
	Hence, if $d(ax_0, x_0)>2r$ we have that $d(anx_0,x_0)>2r$. Thus, for $H\in \mathfrak{a}$, 
	\[\mathcal{A}f(H) = e^{\rho(H)}\int_Nf(\exp(H)n)dm_N(n)\leq 0,\]
	if $\sqrt{\kappa(H, H)} = d(\exp(H)x_0,x_0)>2r$.
	Moreover $\mathcal{F}(\mathcal{A}f) = \mathcal{H}(f)\geq0$ on $\mathfrak{a}^*$.
\end{proof}
\appendix
\section{The spherical Bochner-Schwartz theorem for the Heisenberg space}\label{S:Bochner-Schwarz appendix}
Recall that $\exp:\mathfrak{h}_n\to \mathcal{H}_n$ is a diffeomorphism and that the Schwartz space $\mathscr{S}(\mathcal{H}_n)$ is defined by 
\[\mathscr{S}(\mathcal{H}_n)\coloneqq \{f\circ \exp^{-1}\mid f\in \mathscr{S}(\mathfrak{h}_n)\}.\]
Recall further tht we defined
\[\mathscr{S}(\mathcal{H}_n, U(n))\coloneqq \{f\in \mathscr{S}(\mathcal{H}_n)\mid f(t,kv) = f(t,v)\text{ for all $k\in U(n)$}\}.\]
Our aim in this appendix is a proof of the following theorem.
\begin{theorem}\label{Bochner Schwartz Heisenberg}
	Let $T$ be a positive-definite distribution on $\mathcal{H}_n\rtimes U(n)$.
	Then there exists a unique Borel measure $\mu$ on $PS(\mathcal{H}_n\rtimes U(n), U(n))$ and a continuous functional $\tilde{T}$ on the space $\mathscr{S}(\mathcal{H}_n, U(n))$ with the subspace topology coming from $\mathscr{S}(\mathcal{H}_n)$, such that
	\[
	\tilde{T}f = \int \widehat{f}d\mu
	\]
	for all $f\in \mathscr{S}(\mathcal{H}_n, U(n))$.
	Note that this property uniquely defines $\tilde{T}$ and that the Godement-Plancherel theorem for distributions forces $\mu = \widehat{T}$.
\end{theorem}
For the proof we will use that the image of $\mathscr{S}(\mathcal H_n, U(n))$ under the spherical transform has been characterized by Astengo, Blasio and Ricci in \cite{AdBRSphericalSchwartz}.

\subsection{Schwartz spaces}
For $l\in \N$ let $\Sigma\subset \R^l$ be closed and set $I(\Sigma)=\{f\in\mathscr{S}(\R^l)\mid f|_\Sigma = 0\}$. 
We define $\mathscr{S}(\Sigma)\coloneqq \mathscr{S}(\R^l)/I(\Sigma)$ and equip $\mathscr{S}(\Sigma)$ with the quotient topology.

Given a seminorm $p$ on a vector space $X$ and a linear subspace $M$, we define $p_M$ on $X/M$ by 
\[
p_M([x])\coloneqq \inf \{p(x+y)\mid y\in M\}.
\]
For $m\in\N_0$ we define the Schwartz seminorms 
\[
p^{k}(f)\coloneqq \sup_{x\in \R^l, \abs{\alpha}\leq k} (1+\norm{x})^k\abs*{\frac{\partial^\alpha}{\partial x^\alpha}f(x)}.
\]

\begin{lemma}
	The topology on $\mathscr{S}(\Sigma)$ is the topology induced by the family $\{q^{k}\mid k\in \N_0\}$ of the seminorms $q^{k}\coloneqq p^{k}_{I(\Sigma)}$.
\end{lemma}
\begin{proof}
	This is clear from the definition of quotient topology.
\end{proof}

\begin{remark}
	There is a canonical isomorphism of algebras
	\[
	\mathcal{R}:\mathscr{S}(\Sigma)\to \{f|_\Sigma \mid f\in \mathscr{S}(\R^l)\}\eqqcolon \mathscr{S}(\R^l)|_{\Sigma},\quad [f]\mapsto f|_\Sigma
	\]
	and we will identify these two spaces when convenient (and in particular equip $\mathscr{S}(\R^l)|_{\Sigma}$ with the topology induced by this isomorphism).
	Given $f\in \mathscr{S}(\R^l)|_{\Sigma}$ and $\alpha,\beta\in \N_0^l$ we define
	\[
	\norm{f}_{k}\coloneqq q^{k}(\mathcal{R}^{-1}(f))
	\]
	and note that the topology on $\mathscr{S}(\R^l)|_{\Sigma}$ is induced by the family
	$\{\norm{\cdot}_{k}\}_{k\in \N_0}$.
\end{remark}
\subsection{Embeddings of the Gelfand spectrum}
From now on set $G\coloneqq \mathcal{H}_n\rtimes U(n)$ and $K\coloneqq U(n)$.
We give a more in-depth overview over the harmonic analysis on the Heisenberg group, based on \cite{AdBRSphericalSchwartz}.

	Let $\mathcal{F}_\lambda$ denote the Fock space consisting of entire functions on $\C^n$ such that
	\[
	\norm{F}^2_{\mathcal{F_\lambda}} = \left(\frac{\lambda}{2\pi}\right)^n\int_{\C^n}\abs{F(z)}^2e^{-\frac{\lambda}{2}\abs{z}^2}dz<\infty
	\]
	and define the Bargmann representation $\pi_\lambda$ of $G$ on $\mathcal{F}_\lambda$ by
	\[
	[\pi_\lambda(t,z)F](w)\coloneqq e^{i\lambda t}e^{-\frac{\lambda}{2}\langle w, \overline{z}\rangle-\frac{\lambda}{4}\abs{z}^2}F(w+z)
	\]
	and 
	\[
	\pi_{-\lambda}(t, z) \coloneqq \pi_\lambda(-t, z).
	\]
	The space $\mathcal{P}(\C^n)$ of polynomials on $\C^n$ is dense in $\mathcal{F}_\lambda$ and decomposes under the action of $K$ into $K$-irreducible subspaces 
	\[
	\mathcal{P}(\C^n) = \sum_{\alpha\in \Lambda}P_\alpha,
	\]
	with $\Lambda\subset \widehat{K}$.
	Let $\{v_1^\lambda,\dots, v_{\dim(P_\alpha)}^\lambda\}$ denote an orthonormal basis of $\mathcal{P}(\C^n)$ with respect to the Fock scalar product on $\mathcal{F}_\lambda$ and set
	\begin{align*}
	\phi_{\lambda,\alpha}(t, z)\coloneqq \frac{1}{\dim(P_\alpha)}\sum_{j=1}^{\dim(P_\alpha)}\langle \pi_\lambda(t,z)v_j^\lambda, v_j^\lambda\rangle_{\mathcal{F}_\lambda}.
	\end{align*}
	We also set
	\begin{align*}
		\eta_{Kw}(t, z)\coloneqq \int_K \exp(i\mathrm{Re}(\langle z, kw\rangle ))dm_K(k).
	\end{align*}
	Then 
	\[
	PS(G, K) = \{\eta_{Kw}\mid w\in \C^n\}\cup \{\phi_{\lambda,\alpha}\mid \lambda\in \R^*, \alpha\in \Lambda\}.
	\]
	We note that $\eta_{Kw}$ only depends on the length $\tau=\norm{w}$. By enumerating $\Lambda$ one can obtain the parametrization of the $U(n)$-spherical functions on the Heisenberg group in Theorem \ref{t:parametrization_spherical_heisenberg}.

Let $\mathbb{D}(G/K)$ denote the algebra of $G$-invariant differential operators on $G/K$. A differential operator $D\in \mathbb{D}(G/K)$ is called \emph{homogeneous of degree $m\in \C$}, if 
\[D(f\circ D_r) = r^mD(f)\circ D_r\]
for all $f\in C^\infty(G/K)$ and $r>0$.
\begin{theorem}[Astengo--Blasio--Ricci,\cite{AdBRSphericalSchwartz}]\label{t:differential_operators_heisenberg_group}
	There are differential operators $V_1,\dots, V_s\in \mathbb{D}(G/K)$ such that \[\{V_0\coloneqq-i\partial_t, V_1,\dots, V_s\}\] generate $\mathbb{D}(G/K)$ and 
	\begin{enumerate}[(i)]
		\item each $V_j$ is homogeneous of degree $2m_j$, with $m_j\in \N$ (and $m_0 = 1$),
		\item each $V_j$ is formally self-adjoint and $\widehat{V}_j(\phi_{1,\alpha})\in \N$ for each $\alpha\in \Lambda$,
		\item $\widehat{V}_j(\eta_{Kw}) = \rho_j(w,\overline{w})$ for every $w\in \C^n$, where $\rho_j$ is a nonnegative homogeneous polynomial of degree $2m_j$, nonzero away from the origin,
	\end{enumerate}
	Each element $\phi$ of $\mathrm{PS}(G,K)$ is smooth and a common eigenfunction of $V_1,\dots, V_s$ with real eigenvalues. 
	Moreover each $\phi\in PS(G, K)$ is uniquely determined by its eigenvalues $\widehat{V}_0(\phi),\dots, \widehat{V}_s(\phi)$ with respect to $V_0,\dots, V_s$. 
\end{theorem}
By \cite{AdBRSphericalSchwartz} the map 
\[
\widehat{V}: BS(G, K) \to \R^{s+1},\ \phi \mapsto (\widehat{V}_0(\phi), \dots, \widehat{V}_s(\phi))
\]
is well-defined and a homemorphism onto its image 
\[
\Sigma_n\coloneqq \widehat{V}(\mathrm{BS}(G, K))=\widehat{V}(\mathrm{PS}(G, K)).
\]
$\Sigma_n$ is a closed subset of $\R^{s+1}$ and the Gelfand transform $f\mapsto \widehat{f}$ defines a map 
\[	
\mathcal{G}:\mathcal{S}(\mathcal{H}_n, U(n))\to \mathrm{Map}(\Sigma_n, \C),\ f\mapsto \widehat{f}\circ \widehat{V}^{-1}.
\]
\begin{theorem}[Astengo--Blasio--Ricci, \cite{AdBRSphericalSchwartz}]\label{t:image_Schwartz_space_Heisenberg_group}
	The map
	\[
	\mathcal{G}: \mathscr{S}(\mathcal{H}_n, U(n))\to \mathscr{S}(\R^{s+1})|_{\Sigma_n},\quad f \mapsto \mathcal{G}(f)
	\]
	is a topological isomorphism. 
	More specifically, for each $p\in \N_0$ there exists a $F_p\in \mathscr{S}(\R^{s+1})$ and $q\in \N$, both depending on $p$ sucht that $F_p|_{\Sigma_n} = \widehat{f}$ and $\norm{F_p|_{\Sigma_n}}_{p}\leq C_p \norm{f}_q$ for $C_p > 0$. 
\end{theorem}
\begin{remark} 
	Note that Theorem \ref{t:differential_operators_heisenberg_group} implies that
	\[\widehat{V}(\eta_{Kw}\circ D_r) = \widehat{V}(\eta_{Krw}) = (0, r^{2m_j}\widehat{V}_1(\eta_{Kw}),\dots,r^{2m_s}\widehat{V}_s(\eta_{Kw}))\]
	and
	\[\widehat{V}(\phi_{\lambda,\alpha}) = (\abs{\lambda}\widehat{V}_0(\phi_{1,\alpha}), \abs{\lambda}^{m_1}\widehat{V}_1(\phi_{1,\alpha},\dots, \abs{\lambda}^{m_s}\widehat{V}_s(\phi_{1,\alpha})))\]
	for all $\lambda\neq 0$ and $\alpha\in \Lambda$. As 
	\[\phi_{\lambda, \alpha} = \phi_{1,\alpha}\circ D_{\sqrt{\lambda}}\] for $\lambda>0$ and
	\[\phi_{\lambda, \alpha} = \overline{\phi_{1,\alpha}}\circ D_{\sqrt{\abs{\lambda}}}\] for $\lambda<0$,
	we see that
	\[\widehat{V}(\phi\circ D_r) = (r^2\widehat{V}_0(\phi), r^{2m_1}\widehat{V}_1(\phi),\dots, r^{2m_s}\widehat{V}_s(\phi))\]
	for any $\phi\in PS(G,K)$ and $r>0$.
\end{remark}
\subsubsection{Proof of the spherical Bochner-Schwartz theorem}
\begin{lemma}\label{l:polynomial_growth_Godement_Plancherel_measure_Heisenberg}
	Let $T$ be a positive definite bi-$U(n)$-invariant distribution on $\mathcal{H}_n\rtimes U(n)$ and let $\mu$ denote the Godement-Plancherel measure of $T$.
	Then there exists a positive polynomial $Q$ on $\R^{s+1}$ such that 
	\[
	\int_{\mathrm{PS}(G, K)} \frac{1}{Q\circ\widehat{V}} \ d\mu<\infty
	\]
\end{lemma}
\begin{proof}
	Let $\phi$ be a smooth function on $\mathcal{H}_n\rtimes U(n)$ with $B_{1/2}\subset \mathrm{supp}(\phi)\subset B_{1/2+\delta}$ (where the balls are with respect to the Cygan-Koranyi metric) and set 
	$\chi = \phi*\phi^*$. Then $B_{1}\subset \mathrm{supp}(\chi)$.
	Let $h$ denote the homogeneous dimension of $\mathcal{H}_n$ and
	for $\epsilon\leq 1$ set $\chi_\epsilon = \left(\frac{1}{\varepsilon}\right)^h\chi\circ D_{1/\varepsilon}$.
	It follows that there is some $N\in \N$ and $C>0,\ C'>0$ such that
	\begin{align*}
		T(\chi_\varepsilon) &\leq C \norm{\chi_\varepsilon}_N\\
		&= C\sup\{\abs{\partial^\alpha\chi_\epsilon(t,z)}\mid (t,z)\in \mathcal{H}_n, \alpha\in \mathbb{N}_0^{2n+1},\abs{\alpha}\leq N\}\\
		&\leq C'\epsilon^{-h-2N}\norm{\chi}_N,
	\end{align*}
	where the family $(\norm{\cdot}_N)_{N\in \N}$ of seminorms defined by
	\[\norm{f}_N\coloneqq\sup\{\abs{\partial^\alpha f(x)}\mid x\in \mathcal{H}_n, \alpha\in \mathbb{N}_0^{2n+1},\abs{\alpha}\leq N\}\]
	induces the topology on $C_c^\infty(\mathcal{H}_n)$.
		
	We also have that
	\begin{align*}
		T(\chi_\epsilon) &= \int \widehat{\chi}_\varepsilon d\mu
		= \left(\frac{1}{\epsilon}\right)^h \int \int \chi\circ D_{1/\epsilon}(g)\omega(g^{-1})\,dm_G(g)\,d\mu(\omega)\\
		&= \left(\frac{1}{\epsilon}\right)^h\int \int \left(\frac{1}{\epsilon}\right)^{-h}\chi(g)\omega(D_{\epsilon}(g))\,dm_G(g)\,d\mu(\omega)\\
		&= \int \int \chi(g)\omega(D_{\epsilon}(g))\,dm_G(g)\,d\mu(\omega)\\
		&= \int \widehat{\chi}(\omega\circ D_{\epsilon})d\mu(\omega)\\
		&= \int \mathcal{G}(\chi)(\widehat{V}(\omega\circ D_{\epsilon}))\, d\mu\\
		&= \int \mathcal{G}(\chi)(\epsilon^2\widehat{V}_0(\phi),\epsilon^{2m_1}\widehat{V}_1(\phi),\dots, \epsilon^{2m_s}\widehat{V}_s(\phi))\, d\mu\\
		&= \int \mathcal{G}(\chi)(\epsilon^2x_0,\epsilon^{2m_1}x_1,\dots, \epsilon^{2m_s}x_s)\, d\widehat{V}_*\mu(x_0,\dots, x_s)\\
		&\geq \int_{\left\{x_0^2+\dots x_s^2\leq R^2\frac{1}{\epsilon^2}\right\}} \mathcal{G}(\chi)(\epsilon^2x_0,\epsilon^{2m_1}x_1,\dots, \epsilon^{2m_s}x_s)\, d\widehat{V}_*\mu(x_0,\dots, x_s)\\
		&\geq \mu\left(B\left(0, \frac{R}{\epsilon}\right)\right)\inf\{G(\chi)(\varepsilon^2x_0,\varepsilon^{2m_1}x_1,\dots, \varepsilon^{2m_s}x_s)\mid x_0^2+x_1^2+\dots x_s^2\leq R^2\epsilon^{-2}\}\\
		&\geq \mu\left(B\left(0, \frac{R}{\epsilon}\right)\right) \inf\{G(\chi)(x_0,x_1,\dots, x_s)\mid x_0^2+x_1^2+\dots x_s^2\leq R^2\}\\
		&= K\mu\left(B\left(0, \frac{R}{\epsilon}\right)\right),
	\end{align*}
	with
	\[
	K \coloneqq \inf\{\mathcal{G}(\chi)(x)\mid x\in B(0,R)\}.
	\]
	Since 
	\[
	\mathcal{G}(\chi) = \abs{\mathcal{G}(\phi)}^2 \geq 0
	\]
	and 
	\[
	\mathcal{G}(\chi)(0) = \widehat{\chi}(\mathbbm{1})>0
	\]
	we can choose $R>0$, independently from $\epsilon$, such that $K$ is positive (as $\mathcal{G}(\chi)$ is continuous).
	Substituting $\frac{r}{R}=(1/\epsilon)$ and $h=2n+2$ we get
	\[
	\mu(B(0, r))\leq \frac{C'}{K}\norm{\chi}_N\left(\frac{r}{R}\right)^{(2n+2+2N)} = L r^{2(n+1+N)}
	\]
	for some $L>0$, independent of $r$. This implies the claim.
\end{proof}
\begin{proof}[Proof of the Bochner-Schwartz theorem]
	Consider the map
	\[\tilde{T}:\mathscr{S}(\mathcal{H}_n,U(n))\to \C,\quad  f\mapsto \int \widehat{f}d\mu.\]
	By Lemma \ref{l:polynomial_growth_Godement_Plancherel_measure_Heisenberg} and Theorem \ref{t:image_Schwartz_space_Heisenberg_group}, this map is well-defined.
	More precisely, as for any Schwartz function $F$, any positive polynomial $P$ and any subset $A\subset \R^k$ there is a constant $C>0$ with \[\sup_{x\in A}\abs{F(x)P(x)}\leq C,\]  we have 
	\[\abs{F(x)}\leq \frac{C}{P(x)}.\] This implies that the integral $\int\widehat{f}d\mu$ is well-defined for any $f\in \mathscr{S}(\mathcal{H}_n, U(n))$, as $\mu$ has polynomial growth.
	Choosing a Schwartz function $F_0$ and $q$ as in Theorem \ref{t:image_Schwartz_space_Heisenberg_group} restricting to $\mathcal{G}(f)$, we see that
	\begin{align*}
		\abs{\tilde{T}f}&\leq \int \abs{\widehat{f}}d\mu = \int \abs{\mathcal{G}(f)}d\widehat{V}_*\mu=\int \abs{F_0}d\widehat{V}_*\mu\\
		&\leq \int \frac{\norm{F_0Q}_0}{Q}d\widehat{V}_*\mu \leq \int \frac{1}{Q}d\widehat{V}_*\mu \cdot\norm{Q}_0\cdot C_0\norm{f}_q
	\end{align*}
	and thus we see that $\tilde{T}$ induces a well-defined (continuous) functional on $\mathscr{S}(\mathcal{H}_n, U(n))$ satisfying
	\[\tilde{T}f = \int \widehat{f}d\mu\]
	for all $f\in \mathscr{S}(\mathcal{H}_n, U(n))$.
	Moreover 
	\[\tilde{T}(f^**f) = T(f^**f)\]
	for all $f\in C_c^\infty(\mathcal{H}_n, U(n))$. As these functions form a dense subset of $\mathscr{S}(\mathcal{H}_n, U(n))$ (with respect to the topology on $\mathscr{S}(\mathcal{H}_n)$), given a function $g\in \mathscr{S}(\mathcal{H}_n, U(n))$, we can choose a functions $g_n\in C_c^\infty(\mathcal{H}_n, U(n))$ such that $g_n\to g$ in $C_c^\infty(\mathcal{H}_n)$.
	As the inclusion $C_c^\infty(\mathcal{H}_n)\to \mathscr{S}(\mathcal{H}_n)$ is continuous (as this is true for Euclidean space), we see that $g_n\to g$ in $\mathscr{S}(\mathcal{H}_n)$ and thus $\tilde{T}g = Tg$.
\end{proof}

\bibliographystyle{amsplain}
\providecommand{\bysame}{\leavevmode\hbox to3em{\hrulefill}\thinspace}
\providecommand{\MR}{\relax\ifhmode\unskip\space\fi MR }
\providecommand{\MRhref}[2]{%
  \href{http://www.ams.org/mathscinet-getitem?mr=#1}{#2}
}
\providecommand{\href}[2]{#2}

\end{document}